\documentclass[11pt,a4paper]{article}
\usepackage{amsmath}
\usepackage{amsthm}
\usepackage{amssymb}
\usepackage{amscd}
\usepackage{graphicx}
\usepackage{epsfig}
\usepackage[matrix,arrow,curve]{xy}
\usepackage{color}
\usepackage{mathrsfs}

\usepackage{bbm}
\usepackage{stmaryrd}
\usepackage{hyperref}

\usepackage{latexsym}
\usepackage{amsfonts}
\input xy

\newtheorem{theorem}{Theorem}[section]
\newtheorem{proposition}[theorem]{Proposition}

\newtheorem{corollary}[theorem]{Corollary}

\theoremstyle{definition}
\newtheorem{example}[theorem]{Example}
\newtheorem{remark}[theorem]{Remark}
\newtheorem{definition}[theorem]{Definition}

\font\black=cmbx10 \font\sblack=cmbx7 \font\ssblack=cmbx5 \font\blackital=cmmib10  \skewchar\blackital='177
\font\sblackital=cmmib7 \skewchar\sblackital='177 \font\ssblackital=cmmib5 \skewchar\ssblackital='177
\font\sanss=cmss11 \font\ssanss=cmss9 scaled 900 \font\sssanss=cmss8 scaled 600 \font\blackboard=msbm10
\font\sblackboard=msbm7 \font\ssblackboard=msbm5 \font\caligr=eusm10 \font\scaligr=eusm7 \font\sscaligr=eusm5

\font\bsymb=cmsy10 scaled\magstep2
\def\all#1{\setbox0=\hbox{\lower1.5pt\hbox{\bsymb
       \char"38}}\setbox1=\hbox{$_{#1}$} \box0\lower2pt\box1\;}
\def\exi#1{\setbox0=\hbox{\lower1.5pt\hbox{\bsymb \char"39}}
       \setbox1=\hbox{$_{#1}$} \box0\lower2pt\box1\;}

\def\tx#1{{\fam0\relax#1}}

\newfam\bifam
\textfont\bifam=\blackital \scriptfont\bifam=\sblackital \scriptscriptfont\bifam=\ssblackital

\newfam\blfam
\textfont\blfam=\black \scriptfont\blfam=\sblack \scriptscriptfont\blfam=\ssblack

\newfam\bbfam
\textfont\bbfam=\blackboard \scriptfont\bbfam=\sblackboard \scriptscriptfont\bbfam=\ssblackboard

\newfam\ssfam
\textfont\ssfam=\sanss \scriptfont\ssfam=\ssanss \scriptscriptfont\ssfam=\sssanss
\def\sss#1{{\fam\ssfam\relax#1}}

\newfam\clfam
\textfont\clfam=\caligr \scriptfont\clfam=\scaligr \scriptscriptfont\clfam=\sscaligr

\def\pmb#1{\setbox0\hbox{${#1}$} \copy0 \kern-\wd0 \kern.2pt \box0}
\def\pmbb#1{\setbox0\hbox{${#1}$} \copy0 \kern-\wd0
      \kern.2pt \copy0 \kern-\wd0 \kern.2pt \box0}
\def\pmbbb#1{\setbox0\hbox{${#1}$} \copy0 \kern-\wd0
      \kern.2pt \copy0 \kern-\wd0 \kern.2pt
    \copy0 \kern-\wd0 \kern.2pt \box0}
\def\pmxb#1{\setbox0\hbox{${#1}$} \copy0 \kern-\wd0
      \kern.2pt \copy0 \kern-\wd0 \kern.2pt
      \copy0 \kern-\wd0 \kern.2pt \copy0 \kern-\wd0 \kern.2pt \box0}
\def\pmxbb#1{\setbox0\hbox{${#1}$} \copy0 \kern-\wd0 \kern.2pt
      \copy0 \kern-\wd0 \kern.2pt
      \copy0 \kern-\wd0 \kern.2pt \copy0 \kern-\wd0 \kern.2pt
      \copy0 \kern-\wd0 \kern.2pt \box0}

\mathchardef\za="710B  
\mathchardef\zb="710C  
\mathchardef\zg="710D  
\mathchardef\zd="710E  
\mathchardef\zve="710F 
\mathchardef\zz="7110  
\mathchardef\zh="7111  
\mathchardef\zvy="7112 
\mathchardef\zi="7113  
\mathchardef\zk="7114  
\mathchardef\zl="7115  
\mathchardef\zm="7116  
\mathchardef\zn="7117  
\mathchardef\zx="7118  
\mathchardef\zp="7119  
\mathchardef\zr="711A  
\mathchardef\zs="711B  
\mathchardef\zt="711C  
\mathchardef\zu="711D  
\mathchardef\zvf="711E 
\mathchardef\zq="711F  
\mathchardef\zc="7120  
\mathchardef\zw="7121  
\mathchardef\ze="7122  
\mathchardef\zy="7123  
\mathchardef\zf="7124  
\mathchardef\zvr="7125 
\mathchardef\zvs="7126 
\mathchardef\zf="7127  
\mathchardef\zG="7000  
\mathchardef\zD="7001  
\mathchardef\zY="7002  
\mathchardef\zL="7003  
\mathchardef\zX="7004  
\mathchardef\zP="7005  
\mathchardef\zS="7006  
\mathchardef\zU="7007  
\mathchardef\zF="7008  
\mathchardef\zW="700A  

\newcommand{\be}{\begin{equation}}
\newcommand{\ee}{\end{equation}}

\newcommand{\bea}{\begin{eqnarray}}
\newcommand{\eea}{\end{eqnarray}}
\newcommand{\beas}{\begin{eqnarray*}}
\newcommand{\eeas}{\end{eqnarray*}}
\def\*{{\textstyle *}}

\newcommand{\pa}{\partial}
\newcommand{\ti}{\times}

\newcommand{\ad}{{\rm ad}}

\newcommand{\Ll}{{\pounds}}

\def\la{\langle}
\def\ran{\rangle}

\def\cR{{\mathcal R}}

\def\cO{{\mathcal O}}

\def\cF{\mathcal{F}}

\def\wt{\widetilde}

\def\sJ{\mathsf{J}}

\def\sT{\mathsf{T}}

\def\sj{{\sss j}}

\def\xd{\tx{d}}

\newdir{|>}{%
!/4.5pt/@{|}*:(1,-.2)@^{>}*:(1,+.2)@_{>}}

\newcommand{\g}{\mathfrak{g}}

\newcommand{\La}{\big\langle}
\newcommand{\Ra}{\big\rangle}

\newcommand{\N}{\mathbb{N}}
\newcommand{\Z}{\mathbb{Z}}
\newcommand{\R}{\mathbb{R}}

\newcommand{\Pe}{\mathbb{P}}

\newcommand{\n}{\nabla}

\newcommand{\op}[1]{\!\!\mathop{\rm ~#1}\nolimits}

\newcommand{\id}{\op{id}}

\newcommand{\Ad}{\textnormal{Ad}}

\def\g{\mathfrak{g}}

\def\Rt{{\R^\ti}}
\def\Lt{{L^\ti}}

\DeclareMathOperator{\GL}{GL}

\newcommand{\we}{\wedge}

\newcommand{\hzx}{{\hat\xi}}

\newcommand{\wn}{{\widetilde{N}}}

\newcommand{\mn}{{\medskip\noindent}}

\newcommand{\no}{{\noindent}}

\begin{document}
\title{\bf Reductions: precontact versus presymplectic}
\date{}
\author{\\ Katarzyna  Grabowska$^1$\\ Janusz Grabowski$^2$ 
        \\ \\
         $^1$ {\it Faculty of Physics}\\
                {\it University of Warsaw}\\
                \\$^2$ {\it Institute of Mathematics}\\
                {\it Polish Academy of Sciences}
                }
\maketitle
\begin{abstract}
\noindent We show that contact reductions can be described in terms of symplectic reductions in the traditional Marsden-Weinstein-Meyer as well as the constant rank picture. The point is that we view contact structures as particular (homogeneous) symplectic structures. A group action by contactomorphisms is lifted to a Hamiltonian action on the corresponding symplectic manifold, called the symplectic cover of the contact manifold. In contrast to the majority of the literature in the subject, our approach includes general contact structures (not only co-oriented) and changes the traditional view point: contact Hamiltonians and contact moment maps for contactomorphism groups are no longer defined on the contact manifold itself, but on its symplectic cover. Actually, the developed framework for reductions is slightly more general than purely contact, and includes a precontact and presymplectic setting which is based on the observation that there is a one-to-one correspondence between isomorphism classes of precontact manifolds and certain homogeneous presymplectic manifolds.

\par

\end{abstract}


\section{Introduction}
Reductions of systems due to their symmetries are fundamental tools in almost all areas of physics, having strong links to important questions in the traditional classical mechanics of particles,
rigid bodies, fields, fluids, plasmas, and elastic solids, as well as quantum and relativistic
theories. A model appearance of reductions is that for Hamiltonian systems with the celebrated  Marsden-Weinstein-Meyer theorem, defined for Hamiltonian group actions and with the use of momentum maps (see the historical survey  \cite{Marsden:2001}).

\medskip Hamiltonian systems have been intensively studied recently in the context of contact structures, replacing the symplectic ones in the traditional setting. The roots of contact geometry go back to 1872, when Sophus Lie introduced the concept of contact transformations for studying systems of differential equations, but then contact structures became objects of  intensive studies in a purely geometric context, as well as a tool for numerous applications, like in Gibbs' work on thermodynamics, Huygens' work on geometric optics, Hamiltonian dynamics, fluid mechanics, etc. As a nice  source for the  history of contact geometry and topology with an extended list of references, we recommend \cite{Geiges:2001}, and for a comprehensive presentation of contact geometry, we refer to classical monographs \cite{Arnold:1989,Geiges:2008,Libermann:1987}.

\medskip The standard Hamiltonian formulation describes exclusively isolated systems with reversible dynamics, while real systems are constantly in interaction with the environment, which introduces the phenomena of dissipation and irreversibility. Some of these  phenomena can be described mathematically by means of contact geometry methods. There is a huge list of  publications on contact Hamiltonian geometry, let us mention only a few recent ones describing mechanical dissipative systems, both in the Hamiltonian and Lagrangian setting \cite{Bravetti:2017a,Cruz:2018,deLeon:2019, deLeon:2020,deLeon:2022,Gaset:2020}, but also quantum systems \cite{Ciaglia:2018}.

\medskip Unfortunately, almost all this literature deals with trivial (co-oriented) contact structures, where a global contact form $\zh$ on a manifold $M$ is given.
This is fine for local calculations, but it hides the essence of the corresponding intrinsic geometry.
The problem is that, even in co-orientable cases, dealing with contact Hamiltonian vector fields and the Reeb vector field associated with a particular contact form $\zh$ depends strongly on the choice of $\zh$ and does not make any sense for more general contact structures. The point is that, changing a local contact form defining the contact structure into another  one in the same class of conformal equivalence, we  have to also change the Reeb vector field (the new one is not even in the same conformal class) and the original Hamiltonian defined on the contact manifold. Actually, it is a student exercise to prove  that if, for a contact distribution $C\subset\sT M$,
a tangent vector $v_y\in\sT_yM$ is not in $C_y$, then there is a local contact form $\zh$ inducing $C$, i.e., $C=\ker(\zh)$,  such that $v_y=\cR_\zh(y)$, where $\cR_\zh$ is the  Reeb vector field for $\zh$.

\medskip In \cite{Grabowska:2022} we developed an essentially different picture for contact Hamiltonian mechanics and contact Hamilton-Jacobi theory, which is geometrically intrinsic and valid for general contact structures. The main  change of the view point is that contact structures on $M$ are represented by 1-homogeneous symplectic forms $\zw$ on $\Rt$-principal bundles $\zt:P\to M$, where $\Rt=\R\setminus\{0\}=\GL(1,\R)$ is the multiplicative group of nonzero reals. We call these structures \emph{symplectic $\Rt$-bundles} and they are for us genuine contact structures (see \cite{Bruce:2017,Grabowska:2022,Grabowski:2013}). In  this sense, contact geometry is not an `odd-dimensional version' of symplectic geometry but rather a particular, namely homogeneous, symplectic geometry. To make it precise, let us mention that on every $\Rt$-principal bundle $\zt:P\to M$ with respect to an action of the Lie group $\Rt$,
$$h:\Rt\ti P\to P,\quad h(s,x)=h_s(x),$$
a concept of homogeneity of a tensor field is well defined. Namely, a vector field $X$ (resp., a differential form $\za$) on $P$ is \emph{homogeneous of degree $k\in\Z$} if $(h_s)_*(X)=s^{-k}\cdot X$ (resp., $(h_s)^*(\za)=s^k\cdot\za$). Putting $\zF=h_s$ in the general differential geometric formula
\be\label{gf} \zF^*(i_{\zF_*(X)}\za)=i_X\zF^*(\za),\ee
valid for any diffeomorphism $\zF:N_1\to N_2$, any differential form $\za$ on $N_2$, and any vector field $X$ on $N_1$, we immediately get that $i_X\za$ is $(l+k)$-homogeneous if only $X$ is $k$-homogeneous and $\za$ is $l$-homogeneous. It is also easy to see that then $\xd\za$ is also $l$-homogeneous. Any symplectic $\Rt$-bundle associated with a given contact structure we call its \emph{symplectic cover}. All these covers are canonically isomorphic, but a choice of a particular presentation of a symplectic cover may depend on our purposes. All this can be generalized to \emph{precontact structures} and their \emph{presymplectic covers} being presymplectic $\Rt$-bundles. The corresponding theory is developed in this paper.

\medskip Note that any $\Rt$-principal bundle $P\to M$ can be viewed as an open submanifold $L^\ti_P$ of nonzero vectors in the suitable line bundle $L_P$ with the action of $\Rt$ being the multiplication by (nonzero) reals in $L_P$. The line bundle $L_P$ can be viewed as the line bundle associated with the principal bundle $P$ and the standard action of $\Rt$ on $\R$. In the case of a symplectic $\Rt$-principal bundle corresponding to a contact manifold $(M,C)$, the line bundle $L_P$ is dual to $L^C=\sT M\slash C$.
Contact Hamiltonians in this framework are not functions on $M$ but 1-homogeneous functions on $P$, $H:P\to\R$ (alternatively, sections of the line bundle $L^C$). Since the principal bundle may be  non-trivial, there is generally no way to associate with them functions (Hamiltonians) on $M$.
Canonical examples are contact structures on the bundles $M=\sJ^1L$ of first jets of sections of a line bundle $L\to Q$ which are nontrivial if only $L$ is nontrivial. In this case, $P=\sT^*(L^*)^\ti$ with its canonical symplectic form and the so called  \emph{phase lift} (see \cite{Bruce:2016,Bruce:2017,Grabowska:2022,Grabowski:2009,Grabowski:2013}) of the $\Rt$-principal bundle structure on $(L^*)^\ti=L^*\setminus\{0_M\}$ consisting of non-zero vectors in $L$ (cf. \cite{Grabowska:2022,Grabowski:2013}). Note that this phase lift is not the same as the standard cotangent lift of a group action. In \cite{Grabowski:2013}, it was shown that these canonical examples are the only \emph{linear contact structures}. This result is a contact analog of the well-known fact that every linear symplectic structure on a vector bundle $E\to M$ is equivalent to the canonical symplectic structure on $E\simeq\sT^*M$.

\medskip
Since the symplectic form $\zw$ on a symplectic $\Rt$-bundle $P$ is 1-homogeneous with respect to the $\Rt$-action, 1-homogeneous Hamiltonians are closed with respect to the symplectic Poisson (Lagrange) bracket $\{ H_1,H_2\}_\zw$. This defines also a \emph{Jacobi bracket} of sections of $L^*_P$ which makes $L^*_P$ into a \emph{local Lie algebra} in the sense of Kirillov \cite{Guedira:1984,Kirillov:1976,Marle:1991} or, equivalently, into a \emph{Jacobi bundle} in the sense of Marle \cite{Marle:1991} or a \emph{Kirillov manifold} in the sense of \cite{Bruce:2017,Grabowski:2013}.
For trivial principal (or line) bundles we get a Jacobi structure on $M$, i.e., a Jacobi bracket on the $C^\infty(M)$-module $C^\infty(M)$ in the sense of Lichnerowicz \cite{Dazord:1991,Lichnerowicz:1978}. A standard misunderstanding present in the literature is that the Jacobi bracket on $C^\infty(M)$ is viewed as a bracket on the associative  algebra $C^\infty(M)$, and not on $C^\infty(M)$ as a $C^\infty(M)$-module. The corresponding Jacobi structure on $M$ is therefore understood as a pair $(\zL,\zG)$, where $\zL$ is a bivector field and $\zG$ is a vector field on $M$ (satisfying some additional conditions), that comes from taking the constant function 1 as the basic section for this module, while choosing another basic section leads to other tensors $(\zL,\zG)$ for the same Jacobi bracket.

\medskip\noindent
Note that our symplectic $\Rt$-bundles, being particular instances of \emph{symplectic Liouville manifolds} in the terminology of \cite{Libermann:1987}, can be understood as more advanced symplectizations of contact structures, which (at least for contact forms) are very well known  in contact geometry. The use of the multiplicative group $\Rt$ (which is non-connected) instead of the additive group $\R$ (or the multiplicative group $\R^+$ of positive reals), which appears in traditional symplectizations of contact forms, is crucial for including non-trivial contact structures into the picture. Traditional symplectizations of contact forms produce symplectic $\R^+$-principal bundles (which are always trivializable), which are sometimes called \emph{symplectic cones}. The same remains valid for Jacobi brackets and the corresponding \emph{poissonizations}.

\medskip
Contact reductions, as contact analogs of various symplectic reductions, and related questions (for a pure algebraic approach to Poisson and Jacobi reductions see \cite{Grabowski:1994,Ibort:1997}) were already a subject of studies in numerous papers, let us mention \cite{Albert:1998,Boyer:2000,Dragulete:2002,Geiges:1997,Guillemin:1982,Lerman:2001,
Lerman:2003,Lerman:2004,Loose:2001,Willet:2002}. These reductions are mainly reductions by Lie group actions, and generally use contact forms as the starting point, even  if in some cases (e.g, \cite{Willet:2002}) the final reduced structure does not depend on the choice of the contact form in its conformal (equivalence) class. Consequently, contact moment maps are usually defined on contact manifolds themselves, and associated with actions preserving the contact form, which is completely non-geometrical from the contact geometry point of view. An exception is the paper \cite{Zambon:2006}, which is devoted to generalized concepts of a reduction, and
\cite{Loose:2001}, where a concept of the moment map, although geometrically different, is equivalent to ours, but no symplectic geometry is used and only the zero-value of the moment map is considered (which is the trivial part in our setting). Note finally that contact reductions can be viewed as particular cases of reductions of Jacobi bundles. In \cite{Tortorella:2018}, the authors propose a version of coisotropic reductions of contact structures in the language of Jacobi bundles and Jacobi brackets. Some reductions of Jacobi brackets on functions on manifolds (e.g., trivial Jacobi bundles) are developed in \cite{Mikami:1987,Mikami:1989,Nunes:1989,Nunes:1990}.

\medskip
In this paper, we propose a completely different approach by viewing contact reductions as particular (homogeneous) symplectic reductions on symplectic $\Rt$-bundles. Therefore,  we can use some well-known methods from symplectic geometry, and the reductions are intrinsic from the very beginning, not referring to any auxiliary contact form representing the contact structure. Moreover, the proposed reduction procedure is very broad and includes not only contact analogs of the Marsden-Weinstein-Meyer reduction by a Lie group action, but also constant rank submanifold reductions. Our version of a contact Marsden-Weinstein-Meyer reduction is the following (cf. Theorem \ref{MAIN}).
\begin{theorem}\label{Cred}  Let $(M,C)$ be a contact manifold with a symplectic cover $\zt:P\to M$, let $\zr:G\ti M\to M$ be an action on $M$ of a Lie group $G$ by contactomorphisms, and let $J:P\to\g^*$ be the corresponding contact moment map. Let $\zm\in\g^*$ be a weakly regular value of $J$, so the connected Lie subgroup $G_{[\zm]}^0$ of $G$, corresponding to the Lie subalgebra
$$\g_{[\zm]}^0=\{\zx\in \ker(\zm)\,\big|\,\ad^*_\zx(\zm)\in\R\cdot\zm\}$$
of the Lie algebra $\g$ of $G$, acts on the submanifold $M_\zm=\zt(J^{-1}(\zm))$ of $M$. In particular, $G^0_{[\zm]}=G$ if $\zm=0$ and $G$ is connected.

\mn Suppose additionally that $\sT M_\zm$ is transversal to $C$ and the $G^0_{[\zm]}$-action on $M_\zm$ is free and proper.
Then we have a canonical submersion $\zp:M_\zm\to M(\zm)$ onto the orbit manifold $M(\zm)=M_\zm/G^0_{[\zm]}$, which is canonically a contact manifold with the contact structure $C(\zm)=\sT\zp(C\cap\sT M_\zm)$.
\end{theorem}
\no Note that in the case $\zm=0$ our contact reduction comes down to the standard Marsden-Weinstein-Meyer reduction on the symplectic cover of the contact manifold.

\medskip Actually, in the paper we consider a more general situation when a field of hyperplanes $C$ on a manifold $M$ is only \emph{precontact}, i.e., the corresponding 2-form on $C$ is of constant rank (this form is nondegenerate in the contact case). We associate with precontact structures $(M,C)$ presymplectic $\Rt$-bundles, i.e., principal $\Rt$-bundles $\zt:P\to M$ equipped with a 1-homogeneous presymplectic form $\zw$. Note that we always understand presymplectic forms as closed 2-forms \textbf{of constant rank} and all objects are generally smooth. One exception is distributions understood as just fields of linear subspaces of tangent bundles (non-necessarily smooth). On the other hand, constant rank distributions are always assumed to be \emph{regular}, i.e., of constant rank and smooth. In particular, when we speak about fields of hyperplanes, we always mean smooth fields of hyperplanes. Groups and precontact manifolds are connected, if not declared otherwise.

\medskip With every contact vector vector field on a precontact manifold $(M,C)$, i.e., a vector field whose local flow preserves $C$, we associate a Hamiltonian vector field on the presymplectic cover $(P,\zw)$ of $(M,C)$ and a unique 1-homogeneous Hamiltonian on $P$. In the case of a contact group action, this leads to a uniquely defined equivariant moment map $J:P\to\g^*$, so we can follow the ideas of Meyer, Marsden and Weinstein \cite{Marsden:2001,Meyer:1973} if the reduction by symmetry group is concerned. Note finally that presymplectic reductions have already been studied by some authors (e.g, \cite{Bursztyn:2012,Munoz:1999}), including reductions of Dirac structures \cite{Blankenstein:2001,Brahic:2014}.

\medskip
The paper is organized as follows.
In the next section, we present the basics of our understanding of precontact geometry, together with introducing presymplectic $\Rt$-bundles. Also, fundamental theorems describing relations between precontact and presymplectic settings are proved there. In Section 3, we show a precontact analog of the presymplectic reduction, i.e., precontact-to-contact reduction, while Section 4 is devoted to precontact and contact Hamiltonian dynamics. Different types of submanifolds in a precontact manifold are considered in Section 5, together with a contact analog of the constant rank reduction in symplectic geometry. The precontact analogs of the Marsden-Weinstein-Meyer reduction are proved in Section 6. We end up with concluding remarks, showing in part possible applications and follow-up studies.

\section{Principles of precontact geometry}
For our  picture of contact geometry, we refer  generally to \cite{Bruce:2017,Grabowska:2022,Grabowski:2013}. It is easy to generalize the concepts developed there to a precontact setting.

Let $M$ be a manifold of dimension $m$, and let $C\subset \sT M$ be a \emph{field of hyperplanes} on $M$, i.e., a distribution with $(m-1)$-dimensional fibers (corank 1 distribution). Such a distribution is, at least locally, the kernel of a nonvanishing 1-form $\zh$ on $M$, i.e., $C=\ker(\zh)$. Of course, the 1-form $\zh$ is determined only up to conformal equivalence, since $\ker(\zh')=\ker(\zh)$ if and only if $\zh'=f\zh$, where $f$ is a nowhere-vanishing function. Denote with $\zr^C:\sT M\to L^C$ the canonical projection onto the line bundle $L^C=\sT M/C\to M$, and with $\zn^C:C\ti_MC\to L^C$ the skew-symmetric bilinear map which for vector fields $X,Y$ on $M$, taking values in $C$, reads $\zn^C(X,Y)=\zr^C([X,Y])$.

\begin{definition} The distribution $\ker(\zn^C)$ we will denote simply $\zq(C)$ and call the \emph{characteristic distribution of $C$}.  A hyperplane field $C\subset\sT M$ we call a \emph{precontact structure of rank $(2r+1)$} if the 2-form $\zn^C$ on $C$ is of rank $2r$, i.e., the distribution $\zq(C)$ is regular of rank $(m-2r-1)$. Manifolds equipped with a precontact structure (of rank $(2r+1)$) we will call \emph{precontact manifolds} (of rank $(2r+1)$). Any nonvanishing (local) 1-form which determines a precontact structure $C$ of rank $(2r+1)$ as its kernel we call a \emph{precontact form of rank $(2r+1)$}. If $(2r+1)$ is the dimension of $M$, then any precontact structure of rank $(2r+1)$ we call a \emph{contact structure} and any (local) 1-form $\zh$ such that $C=\ker(\zh)$ we call a \emph{contact form}.
\end{definition}
\begin{remark} Of course, all 1-forms conformally equivalent to a precontact form $\zh$ of rank $(2r+1)$ are also precontact forms of rank $(2r+1)$ and contactomorphisms, i.e., diffeomorphisms respecting the fields of hyperplanes, are possible only between precontact structures of the same rank. Note that a differential 2-form $\zb$ is of (constant) rank $2r$ if and only if $\zb^{r}$ is nowhere-vanishing and $\zb^{r+1}=0$. We will frequently use this fact  in the sequel. We will denote the value of a differential form $\zb$ at point $y$ as $\zb_y$ or $\zb(y)$, depending on editorial needs.
\end{remark}
\begin{proposition}\label{p1} Let $C$ be a field of hyperplanes on a manifold $M$ of dimension $m$, and let $\zh$ be a nonvanishing 1-form on $M$ such that $C=\ker(\zh)$. The following are equivalent:
\begin{description}
\item{(1)} the distribution $C$ is a precontact structure of rank $(2r+1)$;
\item{(2)} the 2-form $\xd\zh$ is of rank $2r$ on $C=\ker(\zh)$;
\item{(3)} the $(2r+1)$-form $\zh\we(\xd\zh)^r$ is nonvanishing and $\zh\we(\xd\zh)^{r+1}=0$.
\item{(4)} the characteristic distribution $\zq(\zh)\subset\sT M$ of $\zh$, defined by
$$\zq(\zh)(y)=\big\{Y\in\ker(\zh(y))\,\big|\,\exists\, a\in\R\ \big[\,i_Y\xd\zh=a\cdot\zh(y)\,\big]\,\big\},$$
is regular of rank $(m-2r-1)$;
\end{description}
In any, thus all, of these cases, $\zq(C)=\zq(\zh)$ and $\zq(C)$  is regular and involutive, so it determines a foliation $\cF_C$ on $M$, whose leaves are maximal integral submanifolds of $\zq(C)$.
\end{proposition}
\begin{proof}
$(1\Leftrightarrow 2)$ In the trivialization of $L^C$ induced by $\zh$,
$$\zh^\#:L^C\to M\ti\R\,,\quad \zh^\#([X(y)])=\left(y,i_{X(y)}\zh\right),$$
we have, for vector fields $X,Y\in C$,
\be\label{zn}\zn^C(X,Y)=\zh([X,Y])=-\xd\zh(X,Y)+X\La\zh,Y\Ra-Y\La\zh,X\Ra=-\xd\zh(X,Y),
\ee
so the rank of $\xd\zh\,\big|_C$ equals the rank of $\zn^C$. Here, $\La\cdot,\cdot\Ra$ denotes the canonical pairing between vectors and covectors.

\mn $(2\Leftrightarrow 3)$ Suppose that $\zh\we(\xd\zh)^r$ is nonvanishing at $y\in M$. Then $(\xd\zh)^r$ is nonvanishing on $C_y$, so the rank of $\xd\zh$ on $C_y=\ker(\zh_y)$ is at least $2r$. If this rank is greater than $2r$, then it is at least $2(r+1)$, so $\zh\we(\xd\zh)^{r+1}\ne 0$ at $y$.
Conversely, if the rank of $\xd\zh$ on $\ker(\zh)$ is $2r$, then clearly $\zh\we(\xd\zh)^r$ is nonvanishing and $\zh\we(\xd\zh)^{r+1}=0$.

\mn $(2\Leftrightarrow 4)$ Let us first show that for each $y\in M$ we have $\zq(\zh)(y)=\ker\left(\xd\zh\,\big|_{C_y}\right)$.
If $Y\in\zq(\zh)(y)$, then $Y\in\ker(\zh(y))$ and $i_Y(\xd\zh)=a\zh(y)$ for some $a\in\R$, so $i_Y\xd\zh$ vanishes on $C_y=\ker(\zh)(y)$, thus belongs to kernel of $\xd\zh\,\big|_C$. Conversely, if $Y\in\ker\left(\xd\zh\,\big|_{C_y}\right)$, then $Y\in\ker(\zh(y))$ and the linear function $i_Y\xd\zh$ on $\sT_yM$ vanishes on $\ker(\zh(y))$. Hence, it is of the form $a\zh(y)$ for some $a\in\R$. Consequently, $\xd\zh\,\big|_{C_y}$ is of rank $2r$ if and only if the dimension of its kernel is $m-2r-1$, which finishes the proof.

Finally, from the proof of $(2\Leftrightarrow 4)$ it follows that $\zq(C)=\zq(\zh)$.
The involutivity of $\zq(C)$ is a well-known fact.
\end{proof}
\no We have the following version of Darboux's Theorem for precontact structures.
\begin{theorem}[Precontact Darboux Theorem]\label{Dt} Let $C$ be a precontact structure of rank $(2r+1)$ on a manifold $M$ of dimension $m$, and let $y_0\in M$. Then there are local coordinates $(z,p_i,q^i,u^j)$ in a neighbourhood $U$ of $y_0$, with $i=1,\dots,r$ and $j=1,\dots,m-2r-1$, such that these coordinates vanish at $y_0$ and $C$ is on $U$ the kernel of the 1-form
\be\label{Dc} \zh=\xd z-p_i\,\xd q^i.\ee
\end{theorem}
\begin{proof} Consider a neighbourhood $U$ of $y_0$ which is equipped with coordinates $(z^l,u^j)$ such that the foliation $\cF_C$ is defined locally by the system of equations $z^l=const$ and
$$\zh=f_l(z,u)\,\xd z^l+g_j(z,u)\,\xd u^j$$
is a local 1-form such that $C=\ker(\zh)$. Of course, $l=1,\dots,2r+1$ and $j=1,\dots,m-2r-1$. Since $\pa_{u^j}\in \zq(C)$, we have $g_j=0$ and we can assume that $f_1=1$, as $\zh\ne 0$. Moreover, the condition $i_{\pa_{u^j}}\xd\zh=a_j\zh$ (Proposition \ref{p1} (4)) implies that
$$\frac{\pa f_l}{\pa u^j}=a_j\, f_l,\quad l=1,\dots,2r+1,\ j=1,\dots,m-2r-1.$$
But for $f_1=1$ we get $a_j=0$, so $\zh$ depends on the coordinates $(z^l)$ only. Since $\zh\ne 0$ and $\xd\zh$ is of rank $2r$ on $\ker(\zh)$, the 1-form $\zh$ is actually a contact form in coordinates $(z^l)$, $l=1\dots,2r+1$. Using the contact Darboux Theorem we get (\ref{Dc}), and it is easy to see that we can additionally require that $y_0=0$ in the Darboux coordinates.

\end{proof}
\no Note that the above theorem does not mean that any precontact form of rank $(2r+1)$ inducing $C$ reads locally as in (\ref{Dc}). Indeed, $\xd\zh$ is of rank $2r$ for $\zh$ as in (\ref{Dc}). On the other hand, $\zh'=e^{u^1}(\xd z-\sum_{i=1}^rp_i\xd q^i)$ is in the same conformal class as $\zh$, but
$$\xd\zh'=e^{u^1}\left(\xd u^1\we\xd z-\sum_{i=1}^r\left(\xd p_i\we\xd q^i+p_i\,\xd u^1\we\xd q^i\right)\right),$$
is of rank $2(r+1)$, not $2r$.
\begin{remark}
We will call the foliation $\cF_C$ on the precontact manifold $(M,C)$ the \emph{characteristic foliation}. Note that the characteristic distribution $\zq(C)$ of a precontact structure $C$ is trivial, i.e., $2r+1=\dim(M)$, if and only if the precontact structure is actually a contact one.
According to the Darboux classification of 1-forms \cite{Darboux:1882} (see also {\cite[Ch. V.4]{Libermann:1987}}), the 1-form (\ref{Dc}) can be characterized as a 1-form of class $(2r+1)$ on $U$, i.e., a 1-form $\zh$ satisfying $\zh\we\xd\zh^r\ne 0$  and $(\xd\zh)^{r+1}=0$.

Note that  precontact structures are understood by many authors as just distributions of corank 1,  and precontact forms as just nonvanishing 1-forms (e.g, \cite{Tortorella:2018,Vitagliano:2018,Zambon:2006}), which is too weak in our opinion. In practice and applications, even such authors put additional requirements that made the structures close to what we call precontact structures and precontact forms. In \cite{deLeon:2019,deLeon:2019a} \emph{precontact forms} of class $(2r+1)$ are defined as in the Darboux' classification, so
that the precontact manifolds in  \cite{deLeon:2019,deLeon:2019a} are trivial examples of the precontact structures in our sense (in which $L^C=\sT M/C$ can be non-trivializable).

In what follows, we will also use the concept of a \emph{presymplectic form} after Souriau \cite{Souriau:1970}: a 2-form $\zw$ is \emph{presymplectic} of rank $2(r+1)$ if $\zw$ is closed and its characteristic distribution $\zq(\zw)$ has constant rank equal to $2(r+1)$. Being automatically involutive, $\zq(\zw)$ defines a foliation $\cF_\zw$ which we call the \emph{characteristic foliation} of $\zw$. Again, some authors consider presymplectic forms simply as closed 2-forms, which is too weak for our purposes. As easily seen, our presymplectic forms are by definition non-zero.
\end{remark}
\begin{definition}
A \emph{contactomorphism} between precontact structures $(M_i,C_i)$, $i=1,2$, is a diffeomorphism $\zf:M_1\to M_2$ such that $\sT\zf(C_1)= C_2$. A \emph{contact vector field} on a precontact manifold $(M,C)$ is a vector field whose local flow consists of contactomorphisms. For a vector field $X$ on $M$ and a distribution $D\subset\sT M$ we will write $X\in D$ if $X$ takes values in $D$.
\end{definition}
The following easy proposition states the properties of contactomorphisms and contact vector fields.
\begin{proposition} A diffeomorphism $\zf:M_1\to M_2$ between precontact structures being locally the kernels of precontact forms $\zh_i$, $i=1,2$, is a contactomorphism if and only if $\zf^*(\zh_2)$ is in the conformal class of $\zh_1$, i.e., $\zf^*(\zh_2)=f\zh_1$ for a nonvanishing function $f$. A vector field $X$ on a precontact manifold $(M,C)$ is a contact vector field if and only if $[X,Y]\in C$ for any vector field $Y\in C$, and if and if and only if $\Ll_X\zh=g\zh$ for any (local) precontact form inducing $C$, where $g$ is a function (not necessary nonvanishing) on $M$.
\end{proposition}

\subsection{Presymplectic $\Rt$-bundles}
Throughout the paper, we will use the following notation: if $L\to M$ is a vector bundle, then we denote with $L^\ti$ the open submanifold in $L$ consisting of non-zero  vectors, i.e., $L^\ti=L\setminus 0_M$, where $0_M$ is the zero-section of $L$. Of course, if $L$ is a line bundle, i.e., the rank of $L$ is 1, then $\Lt\to M$ is canonically an $\Rt$-principal bundle with respect to the multiplication by non-zero reals. In what follows, principal bundles with the structure group $\Rt$ we will call simply \emph{$\Rt$-bundles}. Any precontact structure $(M,C)$ determines a line subbundle of the cotangent bundle $\sT^*M$, namely the annihilator $C^o$ of $C$. Any (local) precontact form $\zh$ determining $C$ induces a local trivialization of $C^o$, represented by the line subbundle $[\zh]\subset\sT^*M$ generated by $\zh$, and given by
\be\label{zhti}I_\zh:M\ti\R\to[\zh]\subset\sT^*M,\quad I_\zh(y,s)=s\cdot\zh(y),\ee
Note that under this trivialization, the pull-back of the canonical symplectic form $\zw_M$ on $\sT^*M$ is
\be\label{sf} \zw_\zh=I^*_\zh(\zw_M)=I^*_\zh(\xd\zvy_M)=\xd\left(I^*_\zh(\zvy_M)\right)=\xd(s\zh)=\xd s\we\zh+s\cdot\xd\zh,\ee
where $\zvy_M$ is the Liouville 1-form on $\sT^*M$.
Since the zero-section of $\sT^*M$ is a Lagrangian submanifold, we will consider exclusively the form $\zw_\zh$ restricted to $M\ti\Rt$, which corresponds to the restriction of $\zw_M$ to $[\zh]^\ti$. Note that $\zt:M\ti\Rt\to M$ is a trivial $\Rt$-bundle, and the opposite of the fundamental vector field of the $\Rt$-action and $1\in\R$, where we understand $\R$ as the Lie algebra of $\Rt$, is $\n=s\pa_s$.

\mn Of course, the map (\ref{zhti}), thus the closed 2-form $\zw_\zh$, is defined for any 1-form $\zh$.
The following proposition describes a relation between the characteristic distribution of $\zw_\zh$,
$$\zq(\zw_\zh)=\big\{X\in\sT(M\ti\Rt)\,\big|\,i_X\zw_\zh=0\big\},$$
and the characteristic distribution of $\zh$,
$$\zq(\zh)(y)=\big\{Y\in\ker(\zh(y))\,\big|\,\exists\, a\in\R\ \big[\,i_Y\xd\zh=a\cdot\zh(y)\,\big]\,\big\}.$$
\begin{proposition}\label{zq}
Let $\zh$ be a 1-form on a manifold $M$. For any $Y\in\sT_y M$ and any $a\in\R$, the vector $X=Y-a\cdot s\,\pa_s\in\sT_{(y,s)}(M\ti\Rt)$ is a characteristic vector of the closed 2-form $\zw_\zh$ if and only if $Y$ is a characteristic vector of the form $\zh$.

In particular, the fibers of the characteristic distribution $\zq(\zw_\zh)$ are projected by $\sT\zt$ onto the fibers of the characteristic distribution $\zq(\zh)$ of $\zh$. The projection $\sT_x\zt:\zq(\zw_\zh)(x)\to\zq(\zh)(\zt(x))$ is an isomorphism of vector spaces if and only if $\zh(\zt(x))$ is nonvanishing, so $a$ is uniquely determined. If $\zh(y)=0$, then $\pa_s$ is a characteristic vector of $\zw_\zh(y,s)$ for all $s\in\Rt$.
\end{proposition}
\begin{proof}
The vector $X=Y-a\cdot s\,\pa_s\in\sT_{(y,s)}(M\ti\Rt)$ is a characteristic vector of $\zw_\zh$ if and only if
$$i_X\zw_\zh=s\,(i_Y\xd\zh)(y)-(i_Y\zh)\xd s-a\cdot s\,\zh(y)=0.$$
This is clearly equivalent to $Y\in\ker(\zh)$ and
$$\left(i_Y\xd\zh-a\zh \right)(y)=0,$$
which means, in turn, that $Y$ is a characteristic vector of $\zh$. The number $a\in\R$ is uniquely determined if and only if $\zh(y)\ne 0$, and arbitrary if $\zh(y)=0$, which implies immediately the final statements.

\end{proof}
\no A fundamental observation which connects precontact geometry with the presymplectic one, is the following (cf. \cite{Bruce:2017,Grabowski:2013}).
\begin{theorem}\label{t1}
A nonvanishing 1-form $\zh$ on a manifold $M$ is a precontact form of rank $(2r+1)$ if and only if the closed 2-form (\ref{sf}) on $M\ti\Rt$ is presymplectic of rank $2(r+1)$ (in particular, $\zw_\zh$ is symplectic on $M\ti\Rt$ if and only if $\zh$ is a contact form). In this case, the characteristic distribution $\zq(\zw_\zh)$ of $\zw_\zh$ is transversal to the fibers of the projection $\zt:M\ti\Rt\to M$.
\end{theorem}
\begin{remark} Before proving, let us explain the term `transversal to the fibres' used in the above theorem. What we mean here and in several other places later on is that at each point the intersection of the distribution and the subspace of vectors vertical with respect to the projection $\tau$ is trivial. This notion of transversality does not agree with the traditional one; we have nevertheless decided to use this terminology for the purposes of this paper. We hope it will not lead to any misunderstanding.
\end{remark}
\begin{proof} ($\Rightarrow$) Let us assume that $\zh$ is precontact of rank $(2r+1)$, which means that $\zh\we(\xd\zh)^r\ne 0$, while $\zh\we(\xd\zh)^{r+1}=0$.  The 2-form $\zw_\zh$ is clearly closed. We have
$$(\zw_\zh)^{r+1}=(r+1)s^{r}\xd s\we\zh\we(\xd\zh)^r+s^{r+1}(\xd\zh)^{r+1}.$$
The two summands are linearly independent and $\xd s\we\zh\we(\xd\zh)^r\ne 0$, so $(\zw_\zh)^{r+1}\ne 0$. We have
\be\label{eq1a}(\zw_\zh)^{r+2}=(r+2)s^{r+1}\xd s\we\zh\we(\xd\zh)^{r+1}+s^{r+2}(\xd\zh)^{r+2}=0.\ee
The first summand vanishes, since $\zh\we(\xd\zh)^{r+1}=0$, and the second vanishes, since
$$(\xd\zh)^{r+2}=\xd\left(\zh\we(\xd\zh)^{r+1}\right)=0.$$
Therefore $(\zw_\zh)^{r+2}=0$, thus $(\zw_\zh)$ is of rank $2(r+1)$.

\mn ($\Leftarrow$) Now we assume that $(\zw_\zh)$ is of rank $2(r+1)$, i.e., $(\zw_\zh)^{r+1}\ne 0$ and $(\zw_\zh)^{r+2}=0$. From (\ref{eq1a}) we get that $\zh\we(\xd\zh)^{r+1}=0$, because the summands there are linearly independent.
It remains to show that $(\zw_\zh)^{r+1}\ne 0$ implies now $\zh\we(\xd\zh)^r\ne 0$. This is, of course, true at points in which the rank of $\xd\zh$ is $<2(r+1)$. Suppose that $(\zh\we(\xd\zh)^r)(y)=0$. Hence, the rank of $\xd\zh$ is $2(r+1)$ at $y$. Let us choose $X_1,\dots,X_{2r+2}\in\sT_yM$ such that
$$(\xd\zh)^{r+1}_y(X_1,\dots,X_{2r+2})\ne 0.$$
We can assume additionally that $X_1,\dots,X_{2r+1}\in\ker(\zh_y)$ and $\La\zh_y,X_{2r+2}\Ra=1$ ($\zh$ is nonvanishing). We have
$$0\ne(\xd\zh)^{r+1}_y(X_1,\dots,X_{2r+2})=\sum_{k=1}^{2r+1}(-1)^{k}(r+1)(\xd\zh)^r_y(X_1,\dots,\hat X_k,\dots,X_{{2r}})\xd\zh_y(X_k,X_{2r+2}).$$
Hence, at least one summand must be $\ne 0$, i.e., for some $k$,
$$(\xd\zh)^r_y(X_1,\dots,\hat X_k,\dots,X_{{2r}})\ne 0.$$
But then
$$(\zh\we(\xd\zh)^r)_y\left(X_{2r+2},X_1,\dots,\hat X_k,\dots,X_{2r}\right)=(\xd\zh)^r_y(X_1,\dots,\hat X_k,\dots,X_{{2r}})\ne 0;$$
a contradiction.

\end{proof}
\begin{remark}
Let us note that Theorem \ref{t1} can be seen as a particular case of \cite[Proposition 3.6]{Vitagliano:2018}, where precontact forms are understood as 1-forms with values in a line bundle. This language is better adapted to the work with general Jacobi bundles.
\end{remark}
\begin{corollary}
 Let $C$ be a field of hyperplanes on a manifold $M$. Then $C$ is a precontact structure of rank $(2r+1)$ if and only if the restriction of the canonical symplectic form $\zw_M$ on $\sT^*M$ to $(C^o)^\ti$ is a presymplectic form of rank $2(r+1)$, where $C^o\subset\sT^*M$ is the annihilator of the subbundle $C\subset\sT M$. In particular, $C$ is a contact structure if and only if $(C^o)^\ti$ is a symplectic submanifold of $\sT^*M$.
\end{corollary}
\no Theorem \ref{t1} shows that with any precontact structure $(M,C)$ of rank $(2r+1)$ there is canonically associated the principal $\Rt$-bundle $P=(C^o)^\ti$, with the $\Rt$-action $s\mapsto h_s$ inherited from the vector bundle structure of $\sT^*M$, $h_s(\zh_y)=s\cdot\zh_y$, and the obvious projection $\zt:P\to M$ inherited from the projection $\zp_M:\sT^*M\to M$. Moreover, $P$ is equipped with a presymplectic form $\zw=\zw_M\,\big|_P$ of rank $2(r+1)$ inherited from the canonical symplectic form $\zw_M$ on $\sT^*M$. Since the canonical symplectic structure $\zw_M$ on $\sT^*M$ is linear, the presymplectic form $\zw$ is 1-homogeneous with respect to the $\Rt$-action, $h^*_s(\zw)=s\cdot\zw$. The characteristic distribution of $\zw$ is transversal to the fibers of $\zt_M:(C^o)^\ti\to M$. An abstract counterpart of such a structure is therefore the following (cf. \cite{Bruce:2016,Grabowski:2013}).
\begin{definition}  A \emph{presymplectic $\Rt$-bundle of rank $(2r+1)$}  is an $\Rt$-bundle $\zt:P\to M$ with respect to an $\Rt$-action
$$h:\Rt\ti P\to P\,,\quad \Rt\ti P\ni(s,x)\mapsto h_s(x)\in P\,,$$
equipped additionally with a presymplectic form $\zw$ of rank $2(r+1)$ which is 1-homogeneous, $(h_s)^*(\zw)=s\cdot\zw$, and whose characteristic distribution $\zq(\zw)$ is transversal to the fibers of $\zt$. Such a structure we will denote $(P,\zt,M,h,\zw)$. With every such a presymplectic $\Rt$-bundle there are canonically associated: a nonvanishing $\Rt$-invariant (0-homogeneous) vector field $\n$, being the opposite $\n=-\wt 1$ of the fundamental vector field $\wt 1$ of the $\Rt$-action,
$$\n(x)=\frac{\xd}{\xd t}\,\Big|_{t=0}(h_{e^t}(x))=\frac{\xd}{\xd t}\,\Big|_{t=1}(h_{t}(x)),$$
and a nonvanishing 1-form $\zvy=i_{\n}\zw$. If $\zw$ is symplectic, then we speak about a \emph{symplectic $\Rt$-bundle}.

\mn An \emph{isomorphism of presymplectic $\Rt$-bundles} $(P_i,\zt_i,M_i,h^i,\zw_i)$, $i=1,2$, is an isomorphism
\be\label{cms}
\xymatrix@C+10pt{
 P_1 \ar[r]^{\wt{\zf}}\ar[d]_{\zt_1} & P_2\ar[d]^{\zt_2} \\
M_1 \ar[r]^{\zf} & M_2}
\ee
of $\Rt$-bundles such that $\wt{\zf}^*(\zw_2)=\zw_1$. A vector field $X$ on a presymplectic $\Rt$-bundle we call an \emph{$\Rt$-presymplectic vector field} if it generates a flow of local automorphisms of the presymplectic $\Rt$-bundle
\end{definition}
\begin{proposition}\label{zvy}
Let $(P,\zt,M,h,\zw)$ be a presymplectic $\Rt$-bundle. Then
\begin{description}
\item{(1)} the 1-form $\zvy$ is a unique 1-homogeneous nonvanishing and semi-basic 1-form such that $\xd\zvy=\zw$ (in particular, $\zw$ is always exact);
\item{(2)} the characteristic distribution $\zq(\zw)\subset\sT P$ of $\zw$ is involutive, $\Rt$-invariant, and induces an $\Rt$-invariant foliation $\cF_\zw$ on $P$, with leaves transversal to the fibers of $\zt$;
\item{(3)} any isomorphism (\ref{cms}) respects $\n_i$ and $\zvy_i$, $i=1,2$,
$$\wt\zf_*(\n_1)=\n_2,\quad\text{and}\quad \wt\zf^*(\zvy_2)=\zvy_1;$$
\item{(4)} a vector field $X$ on $P$ is $\Rt$-presymplectic if and only if $X$ is $\Rt$-invariant (homogeneous of degree 0) and $\Ll_X\zvy=0$.
\end{description}
\end{proposition}
\begin{proof}
(1) \ Since $\n$ is vertical, $\zvy$ is semi-basic. But $\n$ is homogeneous of degree 0 ($\Rt$-invariant) and $\zw$ is homogeneous of degree 1, therefore $\zvy$ is 1-homogeneous. Consequently, $\xd(i_{\n}\zw)=\Ll_{\n}\zw=\zw$. If $\zvy_1$ is another 1-homogeneous and semi-basic potential for $\zw$, then $\xd(\zvy-\zvy_1)=0$. But $\zvy-\zvy_1$ is 1-homogeneous and semi-basic, so
$$\zvy-\zvy_1=\Ll_{\n}(\zvy-\zvy_1)=\xd(i_{\n}(\zvy-\zvy_1))+i_{\n}\xd(\zvy-\zvy_1)=0.$$
The first summand on the right-hand side is 0, because $\zvy-\zvy_1$ is semi-basic.

\mn (2) \ The characteristic distribution is clearly of constant rank and involutive, so it induces an $\Rt$-invariant foliation $\cF_\zw$ whose leaves are maximal integral submanifolds of $\zq(\zw)$.
The characteristic distribution $\zq(\zw)$ is transversal to the fibers of $\zt$ by definition, so are the leaves of the foliation $\cF_\zw$.

\mn (3) \ If now $\wt\zf:P_1\to P_2$ is an isomorphism of presymplectic $\Rt$-bundles, it respects the $\Rt$-actions, so $\wt\zf_*(\n_1)=\n_2$. Since $\wt\zf^*(\zw_2)=\zw_1$, from (\ref{gf}) we get $\wt\zf^*(\zvy_2)=\zvy_1$.

\mn (4) \ Finally, a vector field $X$ on $P$ is $\Rt$-presymplectic if it is $\Rt$-invariant and the flow $\zf_t$ of $X$ consists of (local) automorphisms of $(P,\zt,M,h,\zw)$. This means that
$h_s\circ\zf_t=\zf_t\circ h_s$ and, as we have just shown, $\zf_t^*(\zvy)=\zvy$.
It is a standard task to prove that $h_s\circ\zf_t=\zf_t\circ h_s$ is equivalent to $\sT h_s(X(x))=X(h_s(x))$, thus invariance of $X$, and $\zf_t^*(\zvy)=\zvy$ is equivalent to $\Ll_X(\zvy)=0$.

\end{proof}
\noindent The 1-form $\zvy$ we will call the \emph{Liouville 1-form} for the presymplectic $\Rt$-bundle, and the vector field $\n$ - the \emph{Euler vector field}. For the presymplectic $\Rt$-bundle $(C^o)^\ti$ the Liouville 1-form $\zvy$ is the restriction to $(C^o)^\ti$ of the canonical Liouville 1-form $\zvy_M$ on $\sT^*M$.
\no For simplicity, we will usually identify pull-backs of differential forms on $M$ by $\zt$  with the forms themselves, i.e., we write $\zt^*(\zh)=\zh$.
\begin{proposition}\label{pr1} For a {presymplectic $\Rt$-bundle} $(P,\zt,M,h,\zw)$ of rank $2(r+1)$ and an open submanifold $U\subset M$ there is a canonical one-to-one correspondence between trivializations of the principal bundle $P$ over $U$ and precontact forms $\zh$ of rank $(2r+1)$ on $U$.
The correspondence between $\zh$ on $U$ and $\zvy$ on $U\ti\Rt$ is given by $\zvy(y,s)=s\cdot\zh(y)$, so that
$$\zw=\xd s\we\zh+s\cdot\xd\zh.$$
Hence,
$$C=\ker(\zh)=\sT\zt(\ker(\zvy))$$
is a precontact structure on $U$.
\end{proposition}
\begin{proof}
Since the $\Rt$-action on $U\ti\Rt$ is $h_{s'}(y,s)=(y,s'\cdot s)$, we get $\n=s\,\pa_s$ and $\zvy=s\cdot i_{\pa_s}\zw$. Since $\zvy$ is homogeneous of degree 1, the 1-form $i_{\pa_s}\zw=\zvy/s$ is $\Rt$-invariant and nowhere vanishing, and hence it is a pull-back $\zt^*(\zh)$ of a local nowhere-vanishing 1-form $\zh$ on $U$. We will write simply $\zh$ for $\zt^*(\zh)$ on $U\ti\Rt$. Consequently,
$$\zw=\xd\zvy=\xd(s\cdot\zh)=\xd s\we\zh+s\cdot\xd\zh.$$
Since $\zw$ is presymplectic of rank $2(r+1)$, from Theorem \ref{t1} we deduce that $\zh$ is a precontact form of rank $2r+1$. Of course, $\ker(s\zh)=\ker(\zh)$ for $s\ne 0$. Conversely, if $\zh$ is a precontact form on $U$ such that $\ker(\zt^*(\zh))=\ker(\zvy)$, then $\zvy(x)=F(x)\zh(\zt(x))$ for some nonvanishing function $F$ on $\zt^{-1}(U)$. But $\zvy$ is 1-homogeneous, while $\zh$ is 0-homogeneous, so $F$ is 1-homogeneous. We will show that the map $$\Psi:\zt^{-1}(U)\to U\ti\Rt,\quad \Psi(x)=(\zt(x),F(x))$$
is a local trivialization of the principal bundle $P$. Indeed,
$$\Psi(h_s(x))=\big(\zt(h_s(x)),F(h_s(x))\big)=(\zt(x),s\cdot F(x)).$$
In this trivialization $F(x)=s$, so $\zvy=s\cdot\zh$ and $C=\ker(\zh)$ is a precontact structure on $U$.

\end{proof}
\no It is well known (cf. \cite[Theorem 2.1]{Grabowska:2022}) that for every principal $\Rt$-bundle (or a line bundle) $\zt:P\to M$ we can always find an atlas of local trivializations
$$\zf_\za:\zt^{-1}(U_\za)\to U_\za\ti\R\,,$$
with the transition functions
$$\zf_{\za\zb}:(U_\za\cap U_\zb)\ti\Rt\to (U_\za\cap U_\zb)\ti\Rt\,,\quad
\zf_{\za\zb}(x,s)=(x,\pm s)\,.$$
This is equivalent to the fact that we can always find a 2-sheet cover $p:\tilde M\to M$ such that
the pull-back principal bundle $p^*P$ is trivializable. An analogous fact is true for line bundles.
We get, therefore, from Proposition \ref{p1} the following corollary, which is very useful when dealing with precontact structures.
\begin{corollary}
For every precontact manifold $(M,C)$ we can find an open covering  $(U_\za)$ of $M$ and local precontact forms $\zh_\za$ inducing $C$ on $U_\za$, such that $\zh_\za=\pm \zh_\zb$ on $U_{\za\zb}=U_\za\cap U_\zb$. In other words, there exists a 2-sheet cover $p:\tilde M\to M$ such that the pull-back precontact structure $p^*C$ on $\tilde M$ admits a global precontact form.
\end{corollary}
\no We know already that with every precontact structure $(M,C)$ there is a canonically associated presymplectic $\Rt$-bundle $(C^o)^\ti\subset\sT^*M$, with the projection $\zt=\zp_M$, the $\Rt$-action $h$, and the presymplectic form $\zw$ inherited from the cotangent bundle $\zp_M:\sT^*M\to M$. This presymplectic $\Rt$-bundle we will call the \emph{canonical presymplectic cover} of $(M,C)$. For contact structures, we refer to \emph{canonical symplectic covers}. A fundamental result is the converse of this observation (cf. \cite{Grabowski:2013}).
\begin{theorem}\label{FT}
Any {presymplectic $\Rt$-bundle} $(P,\zt,M,h,\zw)$ of rank $2(r+1)$ induces canonically a precontact structure $C$ of rank $(2r+1)$ on $M$, together with an isomorphism
$$
\xymatrix@C+10pt{
 P \ar[r]^{\zF_P}\ar[d]_{\zt} & (C^o)^\ti\ar[d]^{\zp_M} \\
M \ar[r]^{\id_M} & M}
$$
of presymplectic $\Rt$-bundles. This precontact structure is given by $C=\sT\zt(\ker(\zvy))$, where $\zvy$ is the Liouville 1-form on $P$. In other words, there is a one-to-one correspondence between precontact manifolds $(M,C)$ of rank $(2r+1)$ and isomorphism classes of presymplectic $\Rt$-bundles of rank $2(r+1)$ over $M$. As a principal $\Rt$-bundle, $P$ can be therefore identified with $\big[\left(L^C\right)^*\big]^\ti$, where $L^C=\sT M/C$.
Moreover, the projection $\zt:P\to M$ maps locally diffeomorphically the leaves of the characteristic foliation $\cF_\zw$ onto the leaves of the characteristic foliation $\cF_C$. In other words, the projection $\zt$ is on leaves of $\cF_\zw$ a differentiable covering of the leaves of $\cF_C$.
\end{theorem}
\begin{proof}
Since in Proposition \ref{pr1} the local contact structure was defined as $C=\sT\zt(\ker(\zvy))$, so without any reference to the local precontact form $\zh$, this definition gives a global contact structure on $M$. Moreover, we infer from this proposition that, for every $x\in P$, we have $\zvy(x)=\zt^*(\zF_P(x))$, where $\zF_P(x)\in\sT^*_{\zt(x)}M$.
In particular, $\zp_M\circ\zF_P=\zt$, where $\zp_M:\sT^*M\to M$ is the canonical projection. The map $\zF_P$ takes values in $(C^o)^\ti$. Indeed, since any $X_0\in C_{\zt(x)}$ is of the form $\sT\zt(X)$ for some $X\in\ker(\zvy(x))$, we have
$$\La\zF_P(x),X_0\Ra=\La\zF_P(x),\sT\zt(X)\Ra=\La\zt^*(\zF_P(x)),X\Ra
=\La\zvy(x),X\Ra=0.$$
Moreover, $\zF_P(x)\ne 0$, since $\zvy(x)\ne 0$, and so $\zF_P:P\to(C^o)^\ti$ is a morphism of principal $\Rt$-bundles: as $\zvy$ is 1-homogeneous, we have $\zF_P(h_s(x))=s\cdot\zF_P(x)$. This morphism induces the identity on the base, so it is an isomorphism.
In other words, $\zF_P:P\to(C^o)^\ti$ is a diffeomorphism such that $\zF_P\circ h_s=s\cdot\zF_P$, where $\zt_M\circ\zF_P=\zt$. Of course, the annihilator $C^o\subset\sT^*M$ can be canonically identified with the dual line bundle of $L^C=\sT M/C$.

\mn It remains to show that $\zF^*(\zw_M)=\zw$. We will just show that $\zF^*(\zvy_M)=\zvy$, where we denote for simplicity $\zF=\zF_P$ and $\zvy_M$ for the canonical Liouville 1-form on $\sT^*M$. Let $X\in\sT_xP$. We have
\beas
& \La\zF^*\big(\zvy_M(\zF(x))\big),X\Ra=\La\zvy_M(\zF(x)),\sT\zF(X)\Ra=
\La\zF(x),\sT\zp_M(\sT\zF(x))\Ra \\
&=\La\zF(x),\sT\zt(X)\Ra=\La\zt^*(\zF(x)),X\Ra=\La\zvy(x),X\Ra.
\eeas
We used the coordinate-free definition of $\zvy_M$: for any $p\in\sT^\ast M$ and any vector field $Y$ on $\sT^\ast M$ we have
$$\La\zvy_M(p),Y(p)\Ra=\La p,\sT\zp_M\big(Y(p)\big)\Ra.$$

\mn We already know from the local picture (Theorem \ref{t1}) that the fibers of $\zq(\zw)$ are mapped by $\sT\zt$ isomorphically onto the fibers of $\zq(C)$. This immediately implies that the leaves of $\cF_\zw$ are mapped by $\zt$ locally diffeomorphically into the leaves of $\cF_C$. The image of every leaf of $\cF_\zw$ is, therefore, open in the corresponding leaf of $\cF_C$. It is a student exercise to show that it is also closed, so it is the whole leaf.

\end{proof}

\no Any presymplectic $\Rt$-bundle $\zt:P\to M$ inducing a given precontact structure $C$ on $M$ we will call a \emph{presymplectic cover} of $(M,C)$.
The reader could ask why we do not just stay with the canonical presymplectic cover $(C^o)^\ti$. The point is that in many cases we deal with a structure which is clearly a presymplectic $\Rt$-bundle, while the corresponding contact structure $C$, thus $(C^o)^\ti$, is not explicitly given. This is an easy way to define precontact structures.
\begin{example}
For a manifold $M$, the cotangent bundle  $\sT^*M$ with the zero section removed, i.e., $(\sT^*M)^\ti$ is clearly an $\Rt$-bundle with respect to the multiplication by reals in $\sT^*M$. The canonical symplectic form $\zw_M$ restricted to $(\sT^*M)^\ti$ is still symplectic and 1-homogeneous, so we deal with a symplectic $\Rt$-bundle. According to Theorem \ref{FT}, this defines a canonical contact structure on the projectivized cotangent bundle
$\Pe\sT^*M=(\sT^*M)^\ti/\Rt$. In textbooks, one usually uses much more space to define this contact structure. Our approach, however, can be found for the holomorphic case in \cite[Example 55]{Vitagliano:2016}.
\end{example}
\begin{proposition}\label{cc}
Let $(P_i,\zt_i,M_i,h^i,\zw_i)$ be a presymplectic cover of a precontact structure $(M_i,C_i)$, $i=1,2$. If (\ref{cms}) is an isomorphism of presymplectic $\Rt$-bundles, then $\zf$ is a contactomorphism of the corresponding precontact structures on $M_1$ and $M_2$. Conversely, any contactomorphism $\zf:M_1\to M_2$ is covered by a unique isomorphism
$$
\xymatrix@C+10pt{
 P_1 \ar[r]^{\wt{\zf}}\ar[d]_{\zt_1} & P_2\ar[d]^{\zt_2} \\
M_1 \ar[r]^{\zf} & M_2}
$$
of presymplectic $\Rt$-bundles.
\end{proposition}
\begin{proof} According to Proposition \ref{zvy}, we have $\wt\zf^*(\zvy_2)=\zvy_1$, hence $\sT\wt\zf$ maps $\ker(\zvy_1(x))$ onto $\ker(\zvy_2(\wt\zf(x))$. But $\sT\zt_i$ projects the kernels of $\zvy_i(x_i)$ onto $C_i(\zt_i(x_i))$, thus $\sT\zf(C_1(y))=C_2(\zf(y))$, that proves that $\zf$ is a contactomorphism.

Conversely, for the canonical presymplectic covers $P_i=(C_i^o)^\ti\subset\sT^*M_i$, $i=1,2$, we can obtain $\wt{\zf}$ as $(\sT\zf^{-1})^*:\sT^*M_1\to\sT^*M_2$ restricted to $(C_1^o)^\ti$ ($\sT\zf$ respects the precontact structures, so $(\sT\zf^{-1})^*$ respects their annihilators). It remains to show that the isomorphism $\wt\zf$ of presymplectic $\Rt$-bundles covering the contactomorphism $\zf$ is unique. This is equivalent to proving that vertical (i.e., covering the identity) automorphisms $\wt\zf$ of any presymplectic $\Rt$-bundle $(P,\zt,M,h,\zw)$ are trivial. Indeed, since $\wt\zf$ projects onto identity, it acts trivially on pull-backs $\zt^*(\zh)$ of forms $\zh$ on $M$.
Actually, since $\wt\zf$ is a vertical automorphism of the corresponding $\Rt$-bundle, it must be of the form $\wt\zf(x)=h_{f(\zt(x))}$ for a nowhere-vanishing function $f$ on $M$. But if $f\ne 1$, then $\wt\zf$ cannot preserve $\zw$. This is because if $\wt\zf^*(\zw)=\zw$, then one can prove as above that it also preserves the Liouville 1-form $\zvy=i_{\n}\zw$.
But $\zvy$ is semi-basic and 1-homogeneous, so locally, in coordinates $(s,y)$ associated with a local trivialization, it reads $\zvy(s,y)=s\cdot\zh(y)$ for a (pull-back of) 1-form $\zh$ on $M$
(we know already that it is a local precontact form for the precontact structure). Then, $$s\cdot\zh(y)=\wt\zf^*(s\cdot\zh(y))=(s\circ\wt\zf)\cdot\zh(y)=(f(y)s)\cdot\zh(y),$$
so $f(y)=1$.
\end{proof}
\begin{corollary}\label{cH}
Let $(P,\zt,M,h,\zw)$ be a presymplectic cover of a precontact structure $(M,C)$. Then any action $\zr:G\ti M\to M$ of a Lie group $G$ on $M$ by contactomorphisms can be lifted to a unique action $\wt\zr:G\ti P\to P$ on the presymplectic $\Rt$-bundle $P$ by its automorphisms such that $(\wt\zr)_g=\wt{(\zr_g)}$. Conversely, any $G$-action on the presymplectic cover $P$ by automorphisms projects to a $G$-action on $M$ by contactomorphisms. It follows that there is a canonical one-to-one correspondence between contact vector fields $X^c$ on $M$ and $\Rt$-presymplectic vector fields $X$ on $P$, given by $X^c=\zt_*(X)$.
\end{corollary}
\subsection{Remarks on Poisson $\Rt$-bundles}
Note that the concept of a symplectic $\Rt$-bundle has an obvious generalization, namely a \emph{Poisson $\Rt$-bundle}. The Poisson tensor on $P$ should be $(-1)$-homogeneous (the symplectic Poisson tensor $\zL=\zw^{-1}$ is homogeneous of degree $(-1)$ if $\zw$ is 1-homogeneous). Of course, if the Poisson tensor is invertible, we deal  with a symplectic $\Rt$-bundle, i.e., with a contact structure. The concept of a Poisson $\Rt$-bundle coincides with the concept of Kirillov's local Lie algebra \cite{Kirillov:1976}, or \emph{Jacobi bundles} in the terminology of Marle \cite{Marle:1991}, which provides a proper understanding of Jacobi brackets as local brackets on sections of line bundles. In other words, Jacobi brackets are closely related to homogeneous Poisson brackets, that is pretty well known in the literature (see e.g, \cite{Bruce:2017,Dazord:1991,Grabowski:2003,Grabowski:2004,Guedira:1984,Le:2017,Le:2018,Marle:1991}). In non-singular cases, the Lie algebra of Jacobi brackets completely determines the manifold $M$ \cite{Grabowski:2007}. More information about Jacobi brackets and their generalizations, as well as various Lie brackets on manifolds, can be found in \cite{Cosmo:2025,Grabowski:2003a,Grabowski:2013a,Vitagliano:2018}. We will not go deeper into this subject in this paper.

\section{Precontact-to-contact reduction}
Let now $C$ be a precontact structure of rank $(2r+1)$ on a manifold $M$ of dimension $m$, and let
$\zq(C)=\ker(\zn^C)$ be the characteristic distribution of $C$. Since $\zq(C)$ is regular and involutive, by Frobenius Theorem we conclude that it induces a foliation $\cF_C$ of $M$ by maximal integral submanifolds of $\zq(C)$ which we will call the \emph{characteristic foliation of $C$}. The dimension of leaves of this foliation is the rank of $\zq(C)$, which is $(m-2r-1)$. The assumption that the characteristic foliation is simple means that there is a smooth manifold structure on the space $M_0=M/\cF_C$ of leaves such that the canonical projection $p_0:M\to M_0$,
associating to points of a leaf $F$ this leaf in the space of leaves, is a surjective submersion, i.e., $p_0$ is a smooth fibration. Of course, the dimension of $M_0$ is $(2r+1)$.
\begin{theorem}\label{cr}
If the characteristic foliation $\cF_C$ is simple, then $M_0$ carries a canonical contact structure $C_0$ such that $C_0(p_0(y))=\sT p_0(C(y))$. Moreover, if $\zh$ is a precontact form defined in a neighbourhood of $y\in M$ and inducing $C$, then $p_0^*(\zh_0)=f_{\zh}\zh$ for any contact form $\zh_0$ defined in a neighbourhood of $p_0(y)$ and inducing $C_0$, where $f_{\zh}$ is a nonvanishing function.

\end{theorem}
\begin{proof}
Let us observe first that $C_0(p_0(y))=\sT p_0(C(y))$ is a correct definition of a field of hyperplanes $C_0$ on $M_0$. Since the leaves are connected by definition, it is enough to prove that this definition is correct in a neighbourhood of any $y\in M$, i.e.,
\be\label{p}\sT p_0(C(y'))=\sT p_0(C(y))\quad\text{if}\quad p_0(y')=p_0(y).\ee
Let us fix $y$. As $p_0:M\to M_0$ is a fibration, it is locally trivial in a neighbourhood of $y$, so there are local vector fields $X_1,\dots,X_m$ generating $\sT M$ in a neighbourhood of $y$, which are projectable onto vector fields $Y_1,\dots,Y_m$ on $M_0$. In other words, $Y_i$ and $X_i$ are $p_0$-related, $(p_0)_*(X_i)=Y_i$. We can also assume that $X_1,\dots,X_{m-1}$ span $C$ and $X_{2r+1},\dots,X_{m-1}$ span $\zq(C)$, i.e. tangent spaces of the leaves of $\cF_C$. Then $Y_{2r+1}=\cdots=Y_{m-1}=0$ and $Y_1,\dots,Y_{2r},Y_m$ span $\sT M_0$. Let us show first that the local flow $\zf^i_t$ of each $X_i$, $i=2r+1,\dots,m-1$, preserves $C$. Let $Y$ be a local vector field taking values in $C$. Since
$$\zr^C([X_i,Y])=\zn^C(X_i,Y)=0,\quad\text{for}\quad i=2r+1,\dots,m-1,$$
we get that $[X_i,Y]$ again takes values in $C$, thus the local flow $\zf^i_t$ of each $X_i$ preserves $C$ for each $i=2r+1,\dots,m-1$, i.e., $\sT\zf^i_t(C(y))=C(\zf^i_t(y))$. Hence, if $y'=\zf^i_t(y)$, then
$$\sT p_0(C(y'))=\sT p_0\circ\sT\zf^i_t(y)(C(y))=\sT(p_0\circ\zf^i_t)(C(y))=\sT p_0(C(y))\quad\text{for}\quad i=2r+1,\dots,m-1,$$
since clearly $p_0\circ\zf^i_t=p_0$ for such $i$.
As $X_{2r+1},\dots,X_{m-1}$ span the involutive distribution $\zq(C)$, compositions of local diffeomorphisms from their flows act locally transitively on the leaves of $\cF_C$, which shows
(\ref{p}) for $y'$ in a neighbourhood of $y$ in the leaf containing $y$.

Now, it is easy to see that $C_0$ is a contact structure, since the map $p_0$ kills the kernel of $\zn^C$. More precisely, the surjective morphism of vector bundles $\sT p_0:\sT M\to \sT M_0$ induces
canonically a surjective morphism $p^C:L^C\to L^{C_0}$ by
$$p^C\left(\zr^C(X_y)\right)=\zr^{C_0}\left(\sT p_0(X_y)\right).$$
We have
$$\zn^{C_0}(Y_i,Y_j)=\zr^{C_0}\big([Y_i,Y_j]\big)=\zr^{C_0}\big([(p_0)_*(X_i),(p_0)_*(X_j)]\big)
=\zr^{C_0}\big((p_0)_*([X_i,X_j]\big)=\zr^C\big([X_i,X_j]\big),$$
where $i,j=1,\dots,2r$. But $\zn^C$ is nondegenerate on the subbundle in $C$ spanned by $X_1,\dots,X_{2r}$, so the matrix $\Big(\zr^C\big([X_i,X_j]\big)\Big)$ is non-degenerate, that proves that $\zn^{C_0}$ is non-degenerate.
Finally, it follows from (\ref{gf}) that $\ker(p_0^*(\zh_0))=C$, so $p_0^*(\zh_0)=f\zh$ for a nonvanishing function $f$ (of course, $f$ depends on the choice of $\zh$).

\end{proof}
\begin{definition} The contact structure $C_0$ on $M_0$ we call the \emph{reduced contact structure} and the whole procedure - the \emph{precontact-to-contact reduction}.
\end{definition}
\begin{remark}
Note that the local 1-form $p_0^*(\zh_0)$ on $M$ generating $C$ is automatically of class $(2r+1)$ (cf. Theorem \ref{Dt}). Moreover, Theorem \ref{cr} is very similar to contact coisotropic reductions proposed in \cite[Theorem 14]{deLeon:2019} and {\cite[Proposition 4.2]{Tortorella:2018}}. Analogous reductions for general Jacobi structures are considered in \cite[Section 7.2]{Vitagliano:2018}. However, the situation in this case is much more complicated.
\end{remark}
\no A natural question is now, what is the counterpart of precontact-to-contact reductions on the level of presymplectic covers.
\subsection{Symplectic reduction of presymplectic $\Rt$-bundles}
The well-known symplectic reduction of presymplectic manifolds (or constant rank submanifolds in symplectic manifolds) requires some additional caution in the case of presymplectic $\Rt$-bundles, in order to assure that the resulting symplectic manifold has a compatible $\Rt$-bundle structure.
\begin{theorem}\label{main}
Let $(P,\zt,M,h,\zw)$ be a presymplectic cover of a contact manifold $(M,C)$ of rank $2r+1$. Then the characteristic foliation $\cF_\zw$ is simple if and only if the characteristic foliation $\cF_C$ is simple.

\mn In this case, the manifold $P_0=P/\cF_\zw$ of $\cF_\zw$-leaves carries a canonical structure of a symplectic $\Rt$-bundle $(P_0,\zt_0,M_0,h^0,\zw_0)$ such that
the canonical submersion $p:P\to P_0$ is a submersive morphism of $\Rt$-bundles, i.e., $h^0_s\circ p=p\circ h_s$, and $\zw=p^*(\zw_0)$.

\mn In particular, we have the following commutative diagram for a surjective morphism of $\Rt$-bundles:
\be\label{cms0}
\xymatrix@C+50pt{
P \ar[r]^{p}\ar[d]_{\zt} & P_0\ar[d]^{\zt_0} \\
M \ar[r]^{p_0} & M_0.}
\ee
Moreover, $M_0=M/\cF_C$ and $M_0$ carries a canonical contact structure $C_0$ such that $C_0=\sT p_0(C)$, where $p_0$ is the canonical submersion onto the manifold of leaves of $\cF_C$.
\end{theorem}
\begin{proof}
Suppose $\cF_C$ is simple. According to Theorem \ref{cr}, the manifold $M_0=M/\cF_C$ of $\cF_C$-leaves carries a canonical contact structure $C_0$ such that $C_0=\sT p_0(C)$, where
$p_0:M\to M_0$ is the canonical submersion. Consider the canonical symplectic cover $(P_0,\zt_0,M_0,h^0,\zw_0)$ of the contact manifold $(M_0,C_0)$, the Liouville form $\zvy_0$ on $P_0$, and the pull-back bundle $P_1=p_0^*(P_0)$, with the commutative diagram
$$
\xymatrix@C+50pt{
P_1 \ar[r]^{p_1}\ar[d]_{\zt_1} & P_0\ar[d]^{\zt_0} \\
M \ar[r]^{p_0} & M_0.}
$$
It is easy to see that $P_1$ carries a canonical structure of a principal $\Rt$-bundle with the $\Rt$-action $h^1$ inherited from $P_0$, and the above diagram describes a submersive morphism of principal $\Rt$-bundles.
The 1-form $\zvy_1=p^*_1(\zvy_0)$ is clearly nonvanishing, 1-homogeneous, and $\zw_1=\xd\zvy_1=p_1^*(\zw_0)$. Since $p_1$ is a submersion, $\zw_1$ is a closed form of rank $2(r+1)$ and its characteristic distribution $\zq(\zw_1)$ is mapped by $\sT p_1$ onto the trivial distribution, so $\zq(\zw_1)$ is transversal to the fibers of $\zt_1$. Moreover, since $p_0\circ\zt_1=\zt_0\circ p_1$ and $\sT p_1(\ker(\zvy_1))=\ker(\zvy_0)$, we have
$$\sT p_0\big(\sT\zt_1(\ker(\zvy_1))\big)=\sT\zt_0\big(\sT p_1(\ker(\zvy_1))\big)
=\sT\zt_0(\ker(\zvy_0))=C_0.$$
Hence $\sT\zt_1(\ker(\zvy_1))=C$ and we have just proved that $(P_1,\zt_1,M,h^1,\zw_1)$
is a presymplectic cover of $(M,C)$. But all presymplectic covers are isomorphic, so $(P,\zw)\simeq (P_1,\zw_1)$, thus the characteristic foliation $\cF_\zw$ is simple, since $\cF_{\zw_1}$ is simple by definition.

\mn Conversely, let us assume that $\cF_\zw$ is simple. For a leaf $F$ of $\cF_\zw$, denote with $[F]$ the corresponding point of $P_0$ and with $\zw_0$ the reduced symplectic form on $P_0$. In other words, $F=p^{-1}([F])$, where $p:P\to P_0$ is the canonical surjective submersion onto the manifold of leaves (thus a fibration). From symplectic geometry, we know that $P_0$ is equipped with a unique symplectic form $\zw_0$ such that $p^*(\zw_0)=\zw$. Since the $\Rt$-action $h$ on $P$ maps leaves onto leaves, it induces a smooth $\Rt$-action $h^0$ on $P_0$ by $h^0_s([F])=[h_s(F)]$, i.e., $h^0_s\circ p=p\circ h_s$, and the fibers of $\zt$ are projected by $p$ onto the orbits of $h^0$. Since the leaves of $\cF_\zw$ are transversal to the fibers of $\zt$, the submersion $p$ restricted to any fiber of $\zt$ is a local diffeomorphism.  Moreover, since $h_s^*(\zw)=s\cdot\zw$, we have $(h^0_s)^*(\zw_0)=s\cdot\zw_0$, so $\zw_0$ is 1-homogeneous. This easily implies that every leaf $F$ intersects the fibers at no more than one point, which implies that the action $h^0$ is free.

\mn Indeed, if in a fiber $F$ of $\cF_\zw$ there are two different points of the fiber $\zt^{-1}(y_0)$, say $x_0$ and $x_1=h_{s_0}(x_0)$, $s_0\ne 1$, then $v=p(x_0)=p(x_1)\in P_0$, and $h^0_{s_0}(v)=v$. Since $h_{s_0}(h_s(x_0))=h_s(x_1)$, we have $h^0_{s_0}(h^0_s(v))=h^0_s(v)$, so $h^0_{s_0}$ is the identity on the whole $\Rt$-orbit of $v$.

\mn Let us consider a local trivialization $U\ti\Rt$ of $P$ over a neighbourhood $U$ of $y_0\in M$, which is equipped with coordinates $(z^i,u^j)$ such that $y_0=(0,0)$, and
the foliation $\cF_C$ is defined locally by the system of equations $z^i=const$.
This system of coordinates on $U$ gives a system of coordinates $(z^i,u^j,s)$ in $\wt U=\zt^{-1}(U)=U\ti\Rt$ such that $s\in\Rt$ and $h_{s'}(z^i,u^j,s)=(z^i,u^j,s's)$. There is $a\in\Rt$ such that $x_0=(0,0,a)$. These coordinates, reduced to a sufficiently small neighbourhood of $x_0$, induce local coordinates $(z^i,s)$ in a neighbourhood of $v$ in $P_0$ such that $v=(0,a)$. Note that $z^i$ are invariant with respect to the $\Rt$-action on $P_0$. We denote $z^i$ and $z^i\circ p$ with the same symbol, hoping that it is clear from the context on which manifold we are. On the $\Rt$-orbit of $v$ the diffeomorphism $h^0_{s_0}$ is the identity, so $h^0_{s_0}(z^i,s)=(z^i,f(z)s)$ for $z=(z^i)$ sufficiently close to $0$, and $f(0)=1$. Hence, for the symplectic form $\zw_0$ on $P_0$, with
$$\zw_0(v)=\xd s\we(g_i\xd z^i)+h_{ij}\xd z^i\we\xd z^j,\quad g_i,h_{ij}\in\R,$$
we have
$$(h^0_{s_0})^*(\zw_0(v))=\xd(fs)(0,a)\we(g_i\xd z^i)+h_{ij}\xd z^i\we\xd z^j=\xd s\we(g_i\xd z^i)+h'_{ij}\xd z^i\we\xd z^j.$$
The identity $(h^0_{s_0})^*(\zw_0(v))=s_0\cdot\zw_0(v)$, with $s_0\ne 1$, would imply
$$\xd s\we(g_i\xd z^i)=s_0\cdot\xd s\we(g_i\xd z^i);$$
a contradiction, since $g_j\xd u^j\ne 0$ (otherwise $\zw_0$ would not be symplectic).

\mn We have just proved that the smooth $\Rt$-action $h^0$ on $P_0$ is free. Now, we will show that it is proper. To prove this, it is convenient to use Borel's characterization of properness (see \cite[Thm. 1.2.9 (5)]{Palais:1961}, attributed there to Borel): a $G$-action on a manifold $N$ is proper if and only if for every compact subset $K$ of $N$ the subset
$$(K\big| K)=\{g\in G\,\big|\ g\cdot K\cap K\ne\emptyset\}$$
of $G$ is compact.

So, let us take a compact subset $K\subset P_0$, and a sequence $m_1,m_2,\dots$
of points from $K$ which is dense in $K$. Let $x_1,x_2,\dots$ be a sequence of points in $P$ such that $p(x_i)=m_i$, and let $U_i$ be a neighbourhood of $x_i$ having compact closure $\bar U_i$, for all $i=1,2,\dots$. Since $p$ is an open map, the family of open sets $\{p(U_i)\}_{i\in\N}$ covers $K$, so we can choose a finite covering, say $U_1\cup U_2\cup\cdots\cup U_r$. But then
$\wt K=\bar U_1\cup\bar U_2\cup\cdots\cup\bar U_r$ is a compact set and $K\subset p(\wt K)$. Consequently, $(K\big| K)$ is a closed subset of $(\wt K\big|\wt K)$ which is compact, since the $\Rt$-action on $P$ is proper. Moreover, $(P_0,\zt_0,M_0,h^0,\zw_0)$ is a symplectic $\Rt$-bundle and $p$ is a submersive homomorphism of $\Rt$-bundles. Here, $M_0=P_0/\Rt$. The projection $\zt:P\to M$ commutes with the $\Rt$-action, so induces a smooth map
$p_0:M\to M_0$ defined by $p_0(\zt(x))=\zt_0(p(x))$. Let us observe that
$p_0(\zt(x))=p_0(\zt(x'))$ if and only if $p(x)$ and $p(x')$ belong to the same $\Rt$-orbit in $P_0$, say $p(x')=h^0_r(p(x))$, so if and only if $p(x)=p(h_r(x'))$. This, in turn, is equivalent to the fact that $x$ and $h_r(x')$ belong to the same leaf of $\cF_\zw$, and further that $\zt(x)$ and $\zt(x')$ belong to the same leaf of $\cF_C$.
This means that $M_0$ is a manifold of leaves of $\cF_C$, which finishes the proof.

\end{proof}
\begin{remark}
The reduction (\ref{cms0}) we call the \emph{symplectic reduction of a presymplectic $\Rt$-bundle}. Since any foliation is locally trivial, thus locally simple, symplectic reductions as above can always be done locally, i.e., on presymplectic $\Rt$-bundles $\zt^{-1}(U)\subset P$, where $U$ is an open submanifold of $M$ on which the foliation $\cF_C$ is trivial.
\end{remark}
\no If actions of Lie groups by contactomorphisms are concerned, it is easy to see that such actions `commute' with the contact (equivalently, symplectic) reductions. Namely, as the contact actions of Lie groups on precontact manifolds (equivalently, actions on presymplectic $\Rt$-bundles by automorphisms) respect the characteristic distribution, Theorem \ref{main}, together with Corollary \ref{cH}, immediately imply the following.
\begin{proposition}\label{action}
Let $(P,\zt,M,h,\zw)$ be a presymplectic cover of a precontact manifold $(M,C)$, let $\zr:G\ti M\to M$ be an action of a Lie group $G$ on $M$ by contactomorphisms, and let $\wt\zr:G\ti P\to P$ be the canonical lift to a $G$-action on $P$ by automorphisms. Suppose that the characteristic foliation $\cF_C$ (equivalently, the characteristic foliation $\cF_\zw$) is simple, so we have the contact (symplectic) reduction (\ref{cms0}). Then the $G$-actions $\zr$ and $\wt\zr$ induce canonically actions $\zr^0:G\ti M_0\to M_0$ and $\wt\zr^0:G\ti P_0\to P_0$ by contactomorphisms and automorphisms, respectively, according to the formulae
$$\zr^0_g(p_0(y))=p_0(\zr_g(y))\quad\text{and}\quad \wt\zr^0_g(p(x))=p(\wt\zr_g(x)).$$
In particular, $\zt_0\circ\wt\zr^0_g=\zr^0_g$.
\end{proposition}
\subsection{Precontact and presymplectic reductions}
Sometimes we have to deal with reductions of precontact and presymplectic manifolds with respect to foliations smaller than the characteristic foliations. In particular, we have the following obvious generalization of the standard symplectic reduction of presymplectic manifolds.
\begin{proposition} Let $(P,\zw)$ be a presymplectic manifold of rank $2r$. If $\cF$ is a simple foliation of $P$, whose leaves are integral submanifolds of the characteristic distribution $\zq(\zw)$, then the manifold $P_0=P/\cF$ of leaves of $\cF$ carries a canonical presymplectic form $\zw_0$ of rank $2r$ such that $p^*(\zw_0)=\zw$, where $p:P\to P_0$ is the canonical submersion.
\end{proposition}
In the context of precontact manifolds and their presymplectic covers, we can proceed exactly like in Theorems \ref{cr} and \ref{main} for foliations $\cF_M$ and $\cF_P$ replacing $\cF_C$ and $\cF_\zw$, respectively, where the leaves of $\cF_P$ cover the leaves of $\cF_M$ (thus $\cF_P$ is $\Rt$-invariant), if only their leaves are integral submanifolds of the corresponding characteristic foliations. The proofs are essentially the same; the ranks of the reduced $\zw_0$ and $C_0$ remain the same, only the reduced manifolds $P_0$ and $M_0$ are generally of dimensions
greater than the dimensions obtained for contact and symplectic reductions. Hence, we can formulate the following theorem on precontact/presymplectic reductions.
\begin{theorem}\label{main1}
Let $(P,\zt,M,h,\zw)$ be a presymplectic cover of a precontact manifold $(M,C)$ of rank $(2r+1)$. Then there is a canonical one-to-one correspondence between regular $\Rt$-invariant involutive distributions $\zq_P$ of rank $k$ on $P$ such that $\zq_P\subset \zq(\zw)$, and regular involutive distributions $\zq_M$ of rank $k$ on $M$ such that $\zq_M\subset \zq(C)$, given by $\zq_M=\sT\zt(\zq_P)$. Let $\cF_P$ and $\cF_M$ be the corresponding foliations. Then $\cF_P$ is simple if and only if $\cF_M$ is simple.

\mn In this case, the manifold $P_0=P/\cF_P$ of $\cF_P$-leaves caries a canonical structure of a presymplectic $\Rt$-bundle $(P_0,\zt_0,M_0,h^0,\zw_0)$ of rank $2(r+1)$ such that
the canonical submersion $p:P\to P_0$ is a submersive morphism of $\Rt$-bundles, i.e., $h^0_s\circ p=p\circ h_s$, and $\zw=p^*(\zw_0)$.

\mn In particular, we have the following commutative diagram of a surjective morphism of $\Rt$-bundles:
$$
\xymatrix@C+50pt{
P \ar[r]^{p}\ar[d]_{\zt} & P_0\ar[d]^{\zt_0} \\
M \ar[r]^{p_0} & M_0.}
$$
Moreover, $M_0=M/\cF_M$, and $M_0$ carries a canonical precontact structure $C_0$ of rank $(2r+1)$ such that $C_0=\sT p_0(C)$, where $p_0$ is the canonical submersion onto the manifold of leaves of $\cF_M$.
\end{theorem}

\section{Hamiltonian dynamics on precontact manifolds}
\subsection{Precontact Hamiltonians}
A vector field on a presymplectic manifold $(P,\zw)$ we will call a \emph{Hamiltonian vector field} if there is a function $H$ (called a \emph{Hamiltonian of $X$}) on $P$ such that $i_X\zw=-\xd H$.
It is easy to check that a function $H$ is a Hamiltonian if and only if $H$ is constant on the leaves of the characteristic foliation $\cF_\zw$. Note that there is no guarantee that such functions are not just constants. Moreover, if $\zw$ is not symplectic, then a Hamiltonian $H$ does not determine $X$ uniquely: Hamiltonian vector fields $X_H$ with the same Hamiltonian $H$ differ by a vector field taking values in the characteristic distribution $\zq(\zw)$. On the other hand, we have the following.
\begin{proposition} On every presymplectic manifold $(P,\zw)$ there exists a Hamiltonian vector field $X_H$ for each Hamiltonian $H$ (i.e., a function on $P$ which is constant on the leaves of $\cF_\zw$). Two such Hamiltonian vector fields differ by a vector field taking values in the characteristic distribution $\zq(\zw)$. Moreover, we have a well-defined `Poisson bracket' of Hamiltonians, given by
\be\label{Pb}\{H,H'\}_\zw=X_H(H')=\zw(X_H,X_{H'}),\ee
where $X_H,X_{H'}$ are Hamiltonian vector fields with Hamiltonians $H,H'$.
\end{proposition}
\begin{proof} For any $x_0\in P$ there is a neighbourhood $U$ of $x_0$ and a vector field $X^U_H$ on $U$ such that $i_{X^U_H}\zw(x)=-(\xd H)(x)$ for $x\in U$. This is because we can choose $U$ such that $\cF_\zw$ is trivial on $U$, so we can pass locally to the symplectic reduction $p:U\to U_0=U/\cF_\zw$ such that $\zw=p^*(\zw_0)$ for a symplectic form $\zw_0$ on $U_0$. Then $H=H_0\circ p$ for some Hamiltonian $H_0$ on the symplectic manifold $(U_0,\zw_0)$, and as $X^U_H$ we can take any vector field on $U$ which projects \emph{via} $\sT p$ onto the Hamiltonian vector field $X_{H_0}$ on $U_0$. This is possible, since $p$ is a trivial fiber bundle. Of course, any two such vector fields $X^U_H$ differ by a vector field whose projection on $U_0$ is 0. What we have just proved is that the submanifold $A\subset\sT P$, defined by
$$A=\{X_x\in\sT_xP\,\big|\ i_{X}\zw(x)=-(\xd H)(x),\ x\in P\}$$
is a smooth locally trivial affine subbundle in $\sT P$. By topological reasons (the fibers of $A$ are contractible), there is always a global section $X_H$ of $A$.
Finally, the definition of the Poisson bracket is correct, since (\ref{Pb}) does not depend on the choice of the Hamiltonian vector fields.

\end{proof}
\no From Corollary \ref{cH} we deduce now the following.
\begin{proposition}
Let $(P,\zt,M,h,\zw)$ be a presymplectic cover of a precontact manifold $(M,C)$, and let $X$ be an $\Rt$-presymplectic vector field, i.e., $X$ is an $\Rt$-invariant vector field such that $\Ll_X\zvy=0$. Then $X$ admits a unique 1-homogeneous Hamiltonian $H$ defined by $H=i_X\zvy$. Consequently, any contact vector field $X^c$ on $M$ determines a unique 1-homogeneous Hamiltonian $H$ on $P$ and a unique $\Rt$-invariant Hamiltonian vector field $X$ on $P$ such that $i_X\zw=-\xd H$ and $\zt_*(X)=X^c$.
\end{proposition}
\begin{proof}
Since $X$ and $\zvy$ are homogeneous of degrees 0 and 1, respectively, $H=i_X\zvy$ is homogeneous of degree 1. Moreover,
$$i_X\zw=i_X\xd\zvy=\Ll_X\zvy-\xd(i_X\zvy)=-\xd H,$$
so $H$ is a Hamiltonian for $X$.
If $H_1$ is another 1-homogeneous Hamiltonian for $X$, then $\xd(H-H_1)=0$. But $H-H_1$ is 1-homogeneous, so
$$H-H_1=\Ll_{\n}(H-H_1)=i_{\n}\xd(H-H_1)=0.$$
\end{proof}
\noindent We will call $H=i_X\zvy$ the \emph{contact Hamiltonian} of $X^c$. Of course, a contact Hamiltonian can be associated with many contact vector fields $X^c$.
\begin{remark}
Note that the Hamiltonian $H$ associated with $X^c$ can be described directly by
$H(x)=\zi_{\zs_{X^c}}(x)$, where $\zi_{\zs_{X^c}}$ is the linear function on $\left(L^C\right)^*$ defined by the section $\zs_{X^c}$ of $L^C=\sT M/C$, $\zs_{X^c}(y)=\zr^C(X^c(y))$, where we identify $P$ with $\big[\left(L^C\right)^*\big]^\ti$ (cf. Theorem \ref{FT}).
\end{remark}

\subsection{Contact Hamiltonian mechanics}
In this section, we will concentrate on contact manifolds and their symplectic covers. In this case, any function on the symplectic cover is a Hamiltonian and determines a unique Hamiltonian vector field. Note first that we are using the convention in which the canonical symplectic form $\zw_N$ on $\sT^*N$ is $\zw_N=\xd p_i\we\xd q^i$ in canonical coordinates, i.e., $\zw_N=\xd\zvy_N$ for $\zvy_N=p_i\xd q^i$ being the Liouville 1-form on $\sT^*N$, and Hamiltonian vector fields are uniquely defined by
$$ i_{X_H}\zw=-\xd H.$$
This gives the correct Hamilton's equations, and the map $H\mapsto X_H$ is a morphism of the symplectic Poisson bracket
$$\{ H_1,H_2\}_\zw=X_{H_1}(H_2)=\zw(X_{H_1},X_{H_2})$$ into the Lie bracket of vector fields,
$X_{\{ H_1,H_2\}_\zw}=[X_{H_1},X_{H_2}]$.

\mn The commonly accepted approach to contact Hamiltonian dynamics in the physics literature is constructed almost exclusively only for trivial contact manifolds, i.e., manifolds $M$ equipped with a globally defined contact 1-form $\zh$ (cf. \cite{Bravetti:2017a,deLeon:2019,Esen:2021,Grmela:2014,Mrugala:1991,Rajeev:2008,Shaft:2018}), although a more general approach is also known, especially in the context of Jacobi geometry (see e.g, \cite{Grabowska:2022,Grabowska:2024,Grabowska:2024a, Le:2018,Marle:1991,Tortorella:2017}).

\mn For a real valued function ${\hat H}$ (\emph{contact Hamiltonian}) on a trivial contact manifold $({M},\eta)$, the corresponding \emph{contact Hamiltonian vector field} $X^c_{\hat H}$ is a vector field on $M$ defined as the unique one satisfying
\begin{equation}
i_{X^c_{{\hat H}}}\eta =-{\hat H},\qquad i_{X^c_{{\hat H}}}\xd\eta =\xd {\hat H}-\mathcal{R}({\hat H}) \eta\,,   \label{contact}
\end{equation}%
where $\mathcal{R}$ is the Reeb vector field for $\zh$, i.e., $\cR$ is uniquely determined by $i_\cR\zh=1$ and $i_\cR\xd\zh=0$. In this sense, a \emph{contact Hamiltonian system} is the triple  $({M},\eta,{\hat H})$. Since
\begin{equation}\label{L-X-eta}
\mathcal{L}_{X^c_{{\hat H}}}\eta =
\xd\,i_{X^c_{{\hat H}}}\eta+i_{X^c_{{\hat H}}}\xd\eta= -\mathcal{R}({\hat H})\eta\,,
\end{equation}
$X^c_{\hat H}$ is a contact vector field on $M$ with the conformal factor $\lambda=-\mathcal{R}({\hat H})$.
In this realization, the contact Jacobi bracket of two smooth functions on ${M}$ is defined by
\begin{equation}\label{cont-bracket}
\{\hat  F,\hat H\}_\zh=i_{[X^c_{\hat F},X^c_{\hat H}]}\eta\,.
\end{equation}
\no According to \eqref{L-X-eta}, the flow of a contact Hamiltonian vector field preserves the contact structure, but it does not preserve either the contact one-form nor the Hamiltonian function. Instead, we obtain
$${\mathcal{L}}_{X^c_{\hat H}} \, {\hat H} = - \mathcal{R}({\hat H}) {\hat H}\,.$$
Referring to contact Darboux coordinates $(z,q^i,p_j)$ in which $\zh=\xd z-p_i\xd q^i$, the Hamiltonian vector field determined in \eqref{contact} is computed to be
$$
X^c_{\hat H}=\frac{\partial {\hat H}}{\partial p_i}{\partial_{q^i}}  - \left(\frac{\partial {\hat H}}{\partial q^i} + \frac{\partial {\hat H}}{\partial z} p_i \right)
{\partial_{p_i}} + \left(p_i\frac{\partial {\hat H}}{\partial p_i} - {\hat H}\right){\partial_z},
$$
whereas the contact Jacobi bracket \eqref{cont-bracket} is
$$
\{\hat F,{\hat H}\}_\zh = \frac{\partial \hat F}{\partial q^i}\frac{\partial {\hat H}}{\partial p_i} -
\frac{\partial \hat F}{\partial p_i}\frac{\partial {\hat H}}{\partial q^i} + \left(\hat F  - p_i\frac{\partial \hat F}{\partial p_i} \right)\frac{\partial {\hat H}}{\partial z} -
\left({\hat H}  - p_i\frac{\partial {\hat H}}{\partial p_i}\right)\frac{\partial \hat F}{\partial z}.
$$
So, the Hamilton's equations for ${\hat H}$ read
\begin{equation}\label{conham}
\dot{q}^i= \frac{\partial {\hat H}}{\partial p_i}, \qquad \dot{p}_i = -\frac{\partial {\hat H}}{\partial q^i}-
p_i\frac{\partial {\hat H}}{\partial z}, \quad \dot{z} = p_i\frac{\partial {\hat H}}{\partial p_i} - {\hat H}.
\end{equation}
In our setting, \emph{contact  Hamiltonians} for a contact manifold $(M,C)$ are 1-homogeneous Hamiltonians $H:P\to\R$ on the corresponding symplectic cover $\zt:P\to M$ equipped with a 1-homogeneous symplectic form $\zw$. As we already mentioned, 1-homogeneous Hamiltonians on $P$ can be viewed as sections of the line bundle $L^C\to M$, where $L^C=\sT M/C$.
The section corresponding to a Hamiltonian $H$ we will denote $\zs_H$, and the Hamiltonian corresponding to a section $\zs$ will be denoted $H_\zs$.
Since $\zw$ is 1-homogeneous, the 1-homogeneous Hamiltonians are closed with respect to the symplectic Poisson (Lagrange) bracket $\{ H,H'\}_\zw$. This corresponds to a Jacobi bracket $\{\zs,\zs'\}_J$ on sections of $L^C$ \emph{via}
$$\{\zs_H,\zs_{H'}\}_J=\zs_{\{ H,H'\}_\zw}.$$
The corresponding Hamiltonian vector fields $X_H$ (or $X_\zs$) on $P$ are homogeneous of degree 0, i.e., they are $\Rt$-invariant, therefore they project onto the vector fields $X^c_H=X^c_{\zs_H}=\zt_*(X_H)$ on $M$, called \emph{contact Hamiltonian vector fields}.
Contact Hamiltonian vector  fields are actually contact vector fields, i.e., their flows preserve the contact structure (cf. Proposition \ref{cc}), and $H\mapsto X^c_H$ is a one-to-one correspondence between contact Hamiltonians and contact vector fields.

Note that the above understanding of Hamiltonians as sections of certain line bundles is present already in \cite{Marle:1991} (see also \cite{Le:2018}) and valid for an arbitrary Jacobi bundle.
\begin{example}
Let $P=\R^\ti\ti M$ be the trivial $\R^\ti$-principal bundle with coordinates $(s,y=(y^a))$. Any trivial contact structure on $P$ consists of the symplectic form $\zw_\zh$ associated with a contact 1-form $\zh$ on $M$ and defined by $\zw_\zh=\xd s\we\zh+s\cdot\xd\zh$. If $\cR$ is the Reeb vector field for $\zh$, then $\cR$, viewed as homogeneous vector field of weight 0 on $P$, $\cR(s,y)=\cR(y)$, is the Hamiltonian vector field with respect to $\zw_\zh$ with the Hamiltonian $H=s$, $i_\cR\zw_\zh=-\xd s$. Let us take a function $\hat H$ on $M$ and consider the homogeneous Hamiltonian (the sign depends on conventions)
$$H:P\to\R\,,\quad H(s,y)=-s\cdot\hat H(y)\,.$$
The function $\hat H$ we will call a \emph{reduced contact Hamiltonian} (it is defined only for trivial $\Rt$-bundles). The corresponding Hamiltonian vector field $X_H$ is homogeneous of weight 0, so that
$X_H(s,y)=F(y)\,s\cdot\pa_s+Y(y)$, where $F$ is a (pull-back of a) function on $M$ and $i_Y\xd s=0$, so that the vector field $Y$ can be viewed as tangent to $M$ and therefore identified with the contact Hamiltonian vector field $X^c_H=\zt_*(X_H)$. In other words,
$$ X_H(s,y)=F(y)\,s\,\pa_s+X^c_H(y)\,.$$
We have
$$i_{X_H}\zw_\zh=(sF)\cdot\zh-(i_{X^c_H}\zh)\,\xd s+s\cdot i_{X^c_H}\xd\zh=\hat H\,\xd s+s\,\xd\hat H\,,$$
so that $i_{X^c_H}\zh=-\hat H$ and
$$i_{X^c_H}\xd\zh=\xd\hat H-F\,\zh\,.$$
Contracting both sides with $\cR$, we get $F=\cR(\hat H)$, so
$$X_H(s,y)=\cR(\hat H)(y)\,s\,\pa_s+X^c_H(y),$$
and in this way we reconstructed equations (\ref{contact}) for the reduced contact Hamiltonian $\hat H$. In Darboux coordinates $y=(z,q^i,p_j)$ for $\zh$ we recover the contact Hamilton equations (\ref{conham}).
\end{example}
\section{Contact reductions}
\subsection{Isotropic, coisotropic, and Legendrian submanifolds}
Let us fix a contact manifold $(M,C)$ and its symplectic cover $(P,\zt,M,h,\zw)$.
It is obvious that there is a canonical one-to-one correspondence between submanifolds $N$ of $M$ and $\Rt$-subbundles $\wn$ of $P$, given by $N=\zt(\wn)$ and $\wn=\zt^{-1}(N)$.

\mn Since in our language contact structures are understood as certain  homogeneous symplectic structures, objects in contact geometry are consequently homogeneous objects in symplectic geometry. For instance, a submanifold $N$ of the contact manifold $M$ we call \emph{isotropic} (resp., \emph{coisotropic}, \emph{Legendre}) if the inverse image ${\widetilde{N}}=\zt^{-1}(N)$ is isotropic (resp., coisotropic, Lagrangian) in $(P,\zw)$. These concepts are well known in the literature, even for general Jacobi structures, although often defined in other languages. For some deeper studies on contact coisotropic submanifolds, see e.g, \cite{Le:2018,Rosen:2020,Tortorella:2017,Tortorella:2018}.
\begin{proposition}\label{lag} Let $N$ be a submanifold of a contact manifold $(M,C)$. Then,
\begin{enumerate}
\item $N$ is isotropic if and only if $\sT N\subset C\,\big|_N$. Equivalently, the restriction $\zh\,\big|_N$ of any local contact form $\zh$ associated with $C$ vanishes.
\item $N$ is coisotropic if and only if the contact vector fields $X^c_H$ on $M$, associated with contact Hamiltonians $H:P\to\R$ vanishing on $\wn$, are tangent to $N$.
\item $N$ is a Legendre submanifold if and only if it is simultaneously isotropic and coisotropic.
\end{enumerate}
\end{proposition}
\begin{proof}
1. Let $\zh$ be a local contact form on $U\subset M$ such that $C=\ker(\zh)$. We can assume that $P$ over $U$ is trivial and we can write $\zw$ on $\wt{U}=\zt^{-1}(U)$ as
$$\zw(y,s)=\xd s\we\zh(y)+s\cdot\xd\zh(y),$$
where $y\in U$. Assume that $\wn$ is isotropic in $P$. Since $\pa_s$ is tangent to $\wn$, we get that $i_{\pa_s}\zw$ vanishes on $\sT\wn$. But, for $y_0\in N$, the vector space $\sT_{(y_0,s)}\wn$ is spanned by $\pa_s$ and $Y(y_0)\in\sT_{y_0}N$, so $i_{\pa_s}\zw(y_0,s)=\zh(y_0)$. Hence, $\zh(y_0)$ vanishes on $\sT_{y_0}N$. Conversely, if $\zh$ vanishes on $\sT N$, then taking vectors $Y,Y'\in\sT_{y_0}N$  we get
$$\zw\left(a\cdot\pa_s+Y(y_0),b\cdot\pa_s+Y'(y_0)\right)=
a\cdot\zh\left(Y'(y_0)\right)-b\cdot\zh\left(Y(y_0)\right)+\xd\zh\left(Y(y_0),Y'(y_0)\right)=0.$$
This is because if $\zh\,\big|_N=0$, then $\xd\zh\,\big|_N=0$.

\noindent 2. Since contact Hamiltonians are particular Hamiltonians on $P$, and projectable vector fields tangent to $\wn$ project onto vector fields on $M$ which are tangent to $N$, the implication
`$\Rightarrow$' follows easily. Conversely, suppose that the contact vector fields associated with Hamiltonians vanishing on $\wn$ are tangent to $N$. This implies that the Hamiltonian vector fields $X_H$ on $P$ associated with 1-homogeneous Hamiltonians vanishing on $\wn$ are tangent to $\wn$. We have to show that it is true for all Hamiltonians on $P$ vanishing on $\wn$. Since the statement is local, we can assume that $P=\Rt\ti M$ and $\zw(s,y)=\xd s\we\zh(y)+s\cdot\xd\zh(y)$, where $y\in M$. Consider an arbitrary Hamiltonian $H=H(y,s)$ on $P$, $H(y_0,s)=0$ if $y_0\in N$. For a fixed $s_0\in\Rt$, let $\hat H(y)=H(y,s_0)/s_0$. It is a function on $M$, so $H'(y,s)=s\cdot\hat H(y)$ is a 1-homogeneous Hamiltonian on $P$. Let us take $y_0\in N$. Since
$$\xd H'(y_0,s_0)=\hat H(y_0)\xd s+s_0\xd_y\hat H(y_0)=\xd_yH(y_0,s_0)
=\xd H(y_0,s_0),
$$
the vectors $X_H(y_0,s_0)$ and $X_{H'}(y_0,s_0)$ coincide. But the latter is tangent to $\wn$, which completes the proof.

\noindent 3. It follows from the well-known fact that Lagrangian submanifolds in a symplectic manifold can be characterized as being simultaneously isotropic and coisotropic.

\end{proof}
\no Note that $\wn$ being $\Rt$-invariant is actually an $\Rt$-principal bundle itself, $\zt_N:\wn\to N$, where $\zt_N=\zt\,\big|_\wn$, equipped additionally with a 1-homogeneous closed 2-form $\zw\,\big|_\wn$. Here, we understand $\zw\,\big|_\wn$ as $i_\wn^*(\zw)$, where $i_\wn:\wn\hookrightarrow P$ is the canonical inclusion map. We will write simply $\zt$ instead of $\zt_N$, which should not lead to any misunderstanding.

\mn The above concepts can also be formulated in the precontact/presymplectic case, although the concept of Legendrian/Lagrangian submanifolds in the noncontact/nonsymplectic case is used less frequently. For instance, if $\wn$ is a submanifold in a precontact manifold $(P,\zw)$ of rank $2r$, then
\begin{enumerate}
\item $\wn$ is \emph{isotropic} if $\zw\,\big|_\wn=0$, i.e., $\sT\wn+\zq(\zw)\,\big|_\wn\subset(\sT\wn)^\zw$;
\item $\wn$ is \emph{coisotropic} if $(\sT\wn)^\zw\subset\sT\wn+\zq(\zw)\,\big|_\wn$;
\item $\wn$ is \emph{Lagrangian} if it is isotropic and coisotropic, i.e., $(\sT\wn)^\zw=\sT\wn+\zq(\zw)\,\big|_\wn$,
\end{enumerate}
where $(\sT\wn)^\zw$ denotes the 'orthogonal complement' of $\sT\wn$ with respect to $\zw$, i.e., for $x\in\wn$,
$$(\sT\wn)^\zw(x)=\{X\in\sT_xP\,\big|\ \zw(X,Y)=0\quad \text{for all}\quad Y\in\sT_x\wn\}.$$
Since we always have a local symplectic reduction to a symplectic manifold of dimension $2r$, any Lagrangian submanifold $\wn$ is locally projected by the reduction to a Lagrangian submanifold of the reduced symplectic manifold, so the rank $d$ of the intersection $\sT\wn\cap\zq(\zw)\,\big|_\wn$ is locally constant and $\dim{\wn}=r+d$ (cf. {\cite[Definition 2.2]{Guzman:2010}}). For precontact manifolds we have the following analog of Proposition \ref{lag}.
\begin{proposition} Let $N$ be a submanifold of a precontact manifold $(M,C)$. Then,
\begin{enumerate}
\item $N$ is isotropic if and only if $\sT N\subset C\,\big|_N$.
\item $N$ is coisotropic if and only if the contact vector fields $X^c_H$ on $M$, associated with contact Hamiltonians $H:P\to\R$ vanishing on $\wn$, take values in $\sT_yN+\zq(C)(y)$ for $y\in N$.
\item $N$ is a Legendre submanifold if and only if it is simultaneously isotropic nad coisotropic.
\end{enumerate}
\end{proposition}

\subsection{Constant rank reduction}
For general symplectic reductions, we refer to \cite{Benenti:1983,Benenti:1982,Libermann:1987}. For symplectic reductions by Hamiltonian actions of  Lie groups, to \cite{Abraham:1978,Arnold:1989,Libermann:1987,Marsden:1974,Meyer:1973}.

\no \emph{Constant rank submanifolds} in a symplectic manifold are defined as those submanifolds $\wn$ for which the restriction of the symplectic form to $\wn$ is a closed 2-form of constant rank (recall that we decided to call such 2-forms presymplectic forms). It could suggest that this leads naturally to the definition of constant rank submanifolds of a contact manifold $M$ as those submanifolds $N\subset M$ for which $\wn=\zt^{-1}(N)$ is a constant rank submanifold in $(P,\zw)$, i.e., the closed 2-form $\zw\,\big|_\wn$ is of constant rank.
There is, however, a problem with such a definition, especially if we want to work with reductions, since constant rank $\Rt$-subbundles $\wn$ in symplectic $\Rt$-bundles $(P,\zt,M,h,\zw)$ need not to be presymplectic $\Rt$-bundles themselves. This is because the characteristic distribution of $\zw\,\big|_\wn$ need not be transversal to the fibres of $\zt:\wn\to N$.

\mn Let us start by looking closer at the corresponding characteristic distributions.
Let $(P,\zt,M,h,\zw)$ be a presymplectic cover of a precontact manifold $(M,C)$ of rank $(2r+1)$, with the 2-form $\zn^C:C\ti_MC\to L^C$ on $C$ of rank $2r$, the Euler vector field $\n$, and the Liouville 1-form $\zvy$, so $C=\sT\zt(\ker(\zvy))\subset\sT M$. Let $N$ be a submanifold of $M$.
If we denote $\sT^CN=\sT N\cap C$, then it is clear that the dimension of $\sT^C_yN$ may jump at some points $y\in N$, so generally $\sT^CN$ is not even a smooth distribution.

For $x\in\wn$ and $y\in N$, denote with $\zq(\wn)(x)\subset\sT_xP$ the kernel of $\zw\,\big|_\wn(x)$, and with $\zq(N)(y)\subset\sT M$  the kernel of $\left(\zn^C_N\right)(y)$, where $\zn^C_N=\zn^C\,\big|_{\sT^C N}$. The dimensions of these kernels we denote $k(\wn)(x)$ and $k(N)(y)$, respectively. We know that if a 1-form $\zh$ generates the precontact structure $C=\ker(\zh)$ locally on $U\subset M$, then we have a local trivialization $\wt U=U\ti\Rt$ in which $\zw=\zw_\zh$.
From Proposition \ref{zq} we know that $\sT_x\zt$ maps $\zq(\wn)(x)$ onto $\zq(N)(y)$, where $y=\zt(x)$, so $k(\wn)(x)\ge k(N)(y)$ and these dimensions are equal if and only if $\zh\,\big|_N(y)\ne 0$. The latter is equivalent to $\sT_yN\nsubseteq C_y$, and to the fact that $\zw_\wn$ has no vertical characteristic vectors (the characteristic distribution is transversal to the fibers of $\zt$). A submanifolds $N$ of $M$ we will call \emph{transversal} if $\sT_yN\nsubseteq C_y$ for all $y\in N$. In general it could be $k(\wn)(x)=k(N)(y)+1$ even if $k(\wn)$ is constant or $k(N)$ is constant, which shows that we cannot work with constant rank submanifolds in the traditional meaning.
\begin{example}\label{e1}
Consider the contact form $\zh=\xd z-p\xd q$ on $M=\R^3$ and the submanifold $N\subset M$
defined by the equation $z=0$, so the tangent bundle $\sT N$ is spanned by vector fields $\pa_{q},\pa_{p}$ and the contact form restricted to $N$ reads $\zh_N=-p\xd q$. Hence $\xd(\zh_N)=\xd q\we\xd p$ and $\sT^C_{(p,q)}N$ is spanned by $\pa_p$ if $p\ne 0$, and  $\sT^C_{(0,q)}N=\sT_{(0,q)}N$. Consequenlty, $\zq(N)(p,q)$ is spanned by $\pa_p$ if $p\ne 0$, and is trivial if $p=0$, so the generalized distribution $\zq(N)$ is not regular. However, $\wn$ is of constant rank. Indeed, the submanifold $\wn$ on $P=M\ti\Rt$ with adapted coordinates $(z,p,q,s)$ is defined by the same equation $z=0$. Since the symplectic form reads
$$\zw=\xd s\we\zh+s\cdot\xd\zh=\xd s\we\xd z+\xd q\we(s\xd p+p\xd s),$$
its restriction to $\wn$ is
$$\zw\,\big|_\wn=-p\,\xd s\we\xd q+s\,\xd q\we\xd p=\xd q\we(s\xd p+p\xd s).$$
Its kernel is generated by $s\pa_s-p\pa_p$. It is a nowhere vanishing vector field on $\wn$, so $\zq(\wn)$ is of constant rank 1. Actually, $\wn$ is coisotropic. Indeed, the vector field $X(p,q,s)=s\pa_s-p\pa_p$ is the restriction to $\wn$ of the vector field $s\pa_s-p\pa_p-z\pa_z$ which is the Hamiltonian vector field for the 1-homogeneous Hamiltonian $H(z,p,q,s)=s\cdot z$. Note that $X(0,q,s)=s\pa_s$ is vertical, so $\zq(\wn)(0,q,s)$ is the vertical part of the tangent space $\sT_{(0,q,s)}\wn$.
\end{example}
\begin{example} Consider the standard contact form $\zh=\xd z-p_1\xd q^1-p_2\xd q^2$ on $M=\R^5$. Let a submanifold $N$ in $M$ be defined by the equations $z=0, q^1=q^2$, so we can take $(p_1,p_2,q=q^1+q^2)$ as coordinates on $N$ in which the contact form restricted to $N$ reads $\zh_N=-(p_1+p_2)\xd q$. Hence, $\sT^C_{(p_1,p_2,q)}N$ is spanned by $\pa_{p_1},\pa_{p_2}$ if $p_1+p_2\ne 0$ and equals $\sT N$ if $p_1+p_2=0$. Since $\xd(\zh_N)=\xd q\we\xd p_1+\xd q\we\xd p_2$, the kernel of $\xd(\zh_N)$ as a bilinear form on $\sT^CN$ is in both cases spanned by $\pa_{p_1},\pa_{p_2}$  (cf. (\ref{zn})), so $N$ is of constant rank.
\end{example}
\noindent The above examples  show that even if $\wn$ is of constant rank in $(P,\zw)$, the distribution $\zq(N)\subset\sT N$ can be of non-constant rank, and \emph{vice versa}. As Example \ref{e1} shows, even coisotropic submanifolds can be non-transversal. This is, of course, a very bad situation if we want to use this distribution for a reduction, and suggests  the following definition.
\begin{definition} A submanifold $N$ of a contact manifold $(M,C)$ is of \emph{constant rank} if $N$ is transversal and the characteristic distribution $\zq(N)$ is of constant rank.  An $\Rt$-subbundle $\zt:\wn\to N$ of a symplectic $\Rt$ principal bundle $\zt:P\to M$ we call \emph{constantly transversal} if the characteristic distribution $\zq(\wn)\subset\sT\wn$ has constant rank and it is transversal to the fibers of $\zt$ (has no vertical vectors).
\end{definition}
\no The following is now obvious.
\begin{proposition}
Let $(P,\zt,M,h,\zw)$ be a symplectic cover of a contact manifold $(M,C)$, and let $N$ be a submanifold of $M$. Then $N$ is of constant rank if and only $\wn=\zt^{-1}(N)$ is constantly transversal. This is equivalent to the fact that $(N,\sT^CN)$ is a precontact manifold of rank $(2r+1)$, and to the fact that $(\wn,\zt,N,h,\zw\,\big|_\wn)$ is its presymplectic $\Rt$-bundle of rank $2(r+1)$, where $2r=\dim(N)-k(N)$. In this case, the characteristic distributions $\zq(N)$ and $\zq(\wn)$ are regular and involutive, so they define the corresponding foliations $\cF_N$ and $\cF_\wn$, respectively.
\end{proposition}
\no The foliations $\cF_N$ and $\cF_\wn$ we call the \emph{characteristic foliations} of $N$ and $\wn$, respectively. Theorem \ref{main} immediately implies the following result on contact constant rank reduction.
\begin{theorem}[contact constant rank reduction]
Let $(P,\zt,M,h,\zw)$ be a symplectic cover of a contact manifold $(M,C)$, and let $N$ be a constant rank submanifold of $M$ such that its characteristic foliation $\cF_N$ is simple, with the canonical submersion $p_0:N\to N_0=N/\cF_N$. Then $(N,\sT^CN)$ is a precontact manifold, $p_0$ is a precontact-to-contact reduction, and $\wn=\zt^{-1}(N)$ is a presymplectic $\Rt$-subbundle of  $(P,\zt,M,h,\zw)$, with all structures inherited from the latter. Moreover, the characteristic foliation $\cF_\wn$ is also simple, with the canonical submersion $p:\wn\to\wn_0=\wn/\cF_\wn$, and we have the commutative diagram
$$
\xymatrix@C+50pt{
\wn \ar[r]^{p}\ar[d]_{\zt} & \wn_0\ar[d]^{\zt_0} \\
N \ar[r]^{p_0} & N_0\,,}
$$
in which the right-hand side is a symplectic $\Rt$-bundle  and the horizontal maps form a symplectic reduction of the presymplectic $\Rt$-bundle $\wn$.
\end{theorem}
\section{Contact Marsden-Weinstein-Meyer reduction}
\subsection{Precontact moment maps}
Let $(P,\zt,M,h,\zw)$ be a presymplectic cover of a precontact manifold $(M,C)$ of rank $(2r+1)$, and let
$$\zr:G\ti M\to M,\quad \zr(g,y)=g^c(y),$$
be an action of a  Lie group $G$ on $M$ by contactomorphisms. This induces a homomorphism $\zx\mapsto\zx^c$ of the Lie algebra $\g$ of $G$ into the Lie algebra of contact vector fields on $M$, where $\zx^c$ is the fundamental vector field associated with the action $\zr$ and $\zx\in\g$. The group action $\zr$ is covered by an action (see Corollary \ref{cH})
$$\wt\zr:G\ti P\to P, \quad \wt\zr(g,x)=\wt g(x)$$
of $G$ on $P$ by automorphisms of the presymplectic $\Rt$-bundle structure, i.e.,
$$ (\wt g)^*(\zw)=\zw,\quad h_s\circ\wt g=\wt g\circ h_s,\quad \zt\circ\wt g=g^c.$$
Moreover, any contact vector field $\zx^c$ is covered by a unique $\Rt$-invariant Hamiltonian vector field $\hzx$ on $P$, with 1-homogeneous Hamiltonian $H_\zx=i_\hzx\zvy$ on $P$. Therefore, we do not have the problems with defining the moment map, coming from non-uniqueness
of Hamiltonians in the standard symplectic setting.

\mn Let $\hat\g$ be the Lie algebra of vector fields spanned by $\hzx$ (the infinitesimal action of $G$). It is easy to see that $\hat\g(x)$ is the tangent space to the $G$-orbit $G.x=\{\wt g(x)\,\big|\,g\in G\}$ at $x\in P$,
$$\sT_x(G.x)=\hat\g(x)=\{\hzx(x)\,\big|\, \zx\in\g\}.$$
Of course, $\zt:P\to M$ maps $G$-orbits in $P$ onto $G$-orbits in $M$.
We define the map
$$J:P\to\g^*\,,\quad \La J(x),\zx\Ra=H_\zx(x)=(i_\hzx\zvy)(x),$$
which we will call the \emph{contact moment map} associated with the action $\zr$ of $G$ on $M$ by contactomorphisms.
Actually, in the context of moment maps and reductions, only the Lie algebra of vector fields $\hat\g$ (a $\g$-action) is needed if $G$ is connected.
Note that, in contrast to the majority of existing literature (cf. \cite{Albert:1998,deLeon:2019,Geiges:2001,Lerman:2001,Willet:2002}), our contact moment map is defined on $P$, not on $M$. This moment map is known to be $\Ad$-equivariant. The equivariance means that
$$J(\wt g(x))=\Ad^*_g(J(x))=(\Ad_{g^{-1}})^*(J(x)),$$
so
$$H_\zx(\wt g(x))=H_{\Ad_{g^{-1}}(\zx)}(x).$$
In our case, the moment map is not only equivariant, but it intertwines additionally the corresponding $\Rt$-actions,
\be\label{rt} J(h_s(x))=s\cdot J(x).\ee
Indeed,
$$\La  J(h_s(x)),\zx\Ra=H_\zx(h_s(x))=s\cdot H_\zx(x)=\La s\cdot J(x),\zx\Ra,$$
since $H_\zx$ is 1-homogeneous. For any vector subspace $V\subset\sT_xP$, with $V^\zw$ we denote the 'orthogonal complement' of $V$ with respect to $\zw$, i.e.,
$V^\zw(x)=\{X\in\sT_xP\,\big|\ \zw(X,V(x))=0\}$.
\begin{proposition}
Let $K(J)\subset\sT P$ be the kernel of $\sT J$. Then
\be\label{wr1}
K(J)(x)=\bigcap_{\zx\in\g}\ker(\xd H_\zx(x))=(\hat\g(x))^\zw.
\ee
In particular, $\zq(\zw)\subset K(J)$, so the moment map $J$ is constant along the leaves of the characteristic foliation $\cF_\zw$. In the case when the foliation $\cF_\zw$ (equivalently, $\cF_C$) is simple, the map $J$ induces the moment map $J^0:P_0\to\g^*$ for the induced $G$-action (the $\g$-action is enough) on the reduced symplectic $\Rt$-bundle $P_0=P/\cF_\zw$ (cf. Proposition \ref{action}) by $J^0(p(x))=J(x)$, where $p:P\to P_0$ is the canonical submersion.
\end{proposition}
\begin{proof}
Let $X\in K(J)(x_0)$, and let $\zg:\R\to P$ be a smooth curve representing $X$, i.e., $\zg(0)=x_0$ and $\dot\zg(0)=X$. For any $\zx\in\g$ we have
$$0=\La(\sT_{x_0} J)(X),\zx\Ra=\frac{\xd}{\xd t}\,\Big|_{t=0}\La J(\zg(t)),\zx\Ra=\frac{\xd}{\xd t}\,\Big|_{t=0}H_\zx(\zg(t))
=i_X(\xd H_\zx(x_0))=i_{\hzx(x_0)}\,i_X\zw.
$$
The rest is obvious.
\end{proof}
\no The above proposition shows that, in the case of simple $\cF_C$ (e.g, if $C$ is a contact structure), we can reduce our considerations to the standard Marsden-Weinstein-Meyer Hamiltonian reduction.
\subsection{The precontact Marsden-Weinstein-Meyer theorem}
Our construction will be modelled on the standard Marsden-Weinstein-Meyer reduction, with some necessary adaptation to the case of presymplectic $\Rt$-bundles. We will assume that $G$ is connected. Let us fix $\zm\in\g^*$. The inverse image $P_\zm=J^{-1}(\zm)$, considered in the standard approach, is $\Rt$-invariant only if $\zm=0$, so for $\zm\ne 0$ we consider
$$P_{[\zm]}=J^{-1}([\zm]^\ti)=\bigcup_{s\ne 0}h_s(P_\zm),$$
which is the smallest $\Rt$-subbundle in $P$ containing $P_\zm$. Here, $[\zm]^\ti=\{s\zm\,\big|\,s\ne 0\}$. Of course, $P_{s\zm}=h_s(P_\zm)$ and $P_{[s\zm]}=P_{[\zm]}$ for $s\ne 0$ (cf. (\ref{rt})).

We say that $\zm$ is a \emph{weakly regular value} of $J$ if $P_\zm$ is a submanifold and for every $x\in P_\zm$ we have $K(J)(x)=\ker(\sT_xJ)=\sT_xP_\zm$. Of course, due to (\ref{rt}), $\zm$ is a weekly regular value if and only if $s\zm$ is a weakly regular value, where $s\ne 0$, and any regular value is also weakly regular. In what follows, we always assume that $\zm$ is a weakly regular value. Then $\sT_xP_\zm=K(J)(x)$ and (\ref{wr1}) implies that all Hamiltonians $H_\zx$, $\zx\in\g$, are constant on $P_\zm$, thus on all $P_{s\zm}$, $s\in\Rt$. Moreover, $P_{[\zm]}$ is an $\Rt$-bundle over $M_\zm=\zt(P_{[\zm]})$, and, due to the equivariance of the moment map, the subgroup
$$G_\zm=\{g\in G\,\big|\ \Ad_g^*(\zm)=\zm\}$$
of $G$ acts on $P_\zm$. Denote its Lie algebra with $\g_\zm$, and the distribution spanned by $\hat\zx$ for $\zx\in\g_\zm$ with $\hat\g_\zm$. Since we are interested in the submanifold $P_{[\zm]}$ rather than $P_\zm$, we should consider the subgroup
$$G_{[\zm]}=\{g\in G\,\big|\ \Ad_g^*(\zm)\in[\zm]^\ti\}.$$
It leaves the submanifold $P_{[\zm]}$ invariant and, clearly, $G_{[\zm]}=G_\zm$ if and only if the coadjoint orbit $\cO_\zm$ of $\zm$ intersects $[\zm]^\ti$ only at the point $\zm$. The subgroup $G_{[\zm]}$ is closed in $G$, so it is a Lie subgroup. Denote with
$$\g_{[\zm]}=\{\zx\in \g\,\big|\ \ad_\zx^*(\zm)\in[\zm]\}$$
its Lie algebra. Here, $[\zm]=\R\cdot\zm\subset\g^*$.

\mn Suppose that $\zm\ne 0$. In this case the submanifold $P_{[\zm]}$ of $P$ is foliated by 1-codimensional submanifolds $P_{s\zm}$, $s\ne 0$, which are transversal to the $\zt$-fibers, so $\zt(P_{[\zm]})=\zt(P_\zm)=M_\zm$ is a submanifold in $M$, and the subgroup $G_{[\zm]}$ acts on $M_\zm$. We will describe the characteristic distribution of the submanifold $P_{[\zm]}$, i.e., the characteristic distribution $\zq\big(\zw_{[\zm]}\big)$,
where $\zw_{[\zm]}$ is the restriction of $\zw$ to $P_{[\zm]}$. We have
$$\sT P_{[\zm]}=\sT P_\zm\oplus [\n],$$
where $[\n]$ is the line subbundle of vertical vectors in $\sT P_{[\zm]}$, thus generated by the Euler vector field $\n$. Hence,
$$\big(\sT_xP_{[\zm]}\big)^\zw=\big(\sT_xP_\zm\big)^\zw\cap\big(\n(x)\big)^\zw.$$

\mn We know that, for $x\in P_\zm$, the tangent space $\sT_x P_\zm$
equals $K(J)(x)=(\hat\g(x))^\zw$ (see (\ref{wr1})). Assume first that $\zw$ is symplectic, so
$(\sT_xP_\zm)^\zw=\hat\g(x)$. Moreover, $[\n(x)]^\zw=\ker\big(\zvy(x)\big)$, and for $x\in P_{[\zm]}$ we have $$\hat\zx(x)\in\ker\big(\zvy(x)\big)\ \Leftrightarrow\ H_\zx(x)=0\ \Leftrightarrow\ \La J(x),\zx\Ra=0\
\Leftrightarrow\ \zx\in\ker(\zm).
$$
Moreover, $\hat\zx(x)\in\sT_xP_{[\zm]}$ if and only if $\sT J\big(\hat\zx(x)\big)\in[\zm]$. Putting $J(x)=s\zm$, $s\ne 0$, we get
$$\sT J\big(\hat\zx(x)\big)=\frac{\xd}{\xd t}\,\Big|_{t=0}J\big(\wt{\exp(t\zx)}(x)\big)=\frac{\xd}{\xd t}\,\Big|_{t=0}\Ad^*_{\exp(t\zx)}(s\zm)=s\cdot\ad^*_\zx(\zm),
$$
so $\hat\zx(x)\in\sT_xP_{[\zm]}$ if and only $\zx\in\g_{[\zm]}$. Consequently,
$$\zq(\zw_{[\zm]})(x)=\hat\g^0_{[\zm]}(x),$$
where
$$\g^0_{[\zm]}=\{\zx\in \g_{[\zm]}\,\big|\ \zx\in\ker(\zm)\}.$$

\mn It is easy to see that $\g^0_{[\zm]}$ is an ideal in the Lie algebra $\g_\zm$. Indeed, $\zx\in\g_{[\zm]}$ if and only if $\ad^*_\zx(\zm)\in[\zm]$. Hence, for $\zx\in\g_{[\zm]}$, $\zx'\in\g^0_{[\zm]}$ we have
$$0=\La\ad^*_\zx(\zm),\zx'\Ra=-\La\zm,[\zx,\zx']\Ra,$$
i.e., $[\g_{[\zm]},\g^0_{[\zm]}]\subset\g^0_{[\zm]}$.
The corresponding connected normal Lie subgroup in $G_{[\zm]}$ we denote $G^0_{[\zm]}$. We conclude that, in the case of symplectic $\zw$, we have $\zq(\zw_{[\zm]})=\hat\g^0_\zm$. In the general pre-contact case,
$$\zq(\zw_{[\zm]})=\hat\g^0_{[\zm]}+\zq(\zw).$$

\mn Let us pass to the case $\zm=0$. Now, $P_0=P_{[0]}$ is automatically an $\Rt$-subbundle in $P$
over a certain submanifold $M_0=\zt(P_0)$ of $M$, and , and $G_0=G_{[0]}=G$. Similarly to above, we have
$$\zq(\zw_0)(x)=\hat\g(x)+\zq(\zw)(x).$$
Summarizing these observations, we get the following.
\begin{theorem}\label{m1} For any weekly regular value $\zm\in\g^*$ of the moment map $J$, the characteristic distribution $\zq(\zw_{[\zm]})$ of the closed 2-form $\zw_{[\zm]}=\zw\,\big|_{P_{[\zm]}}$ on $P_{[\zm]}$ is
$$ \zq(\zw_{[\zm]})=\hat\g^0_{[\zm]}+\zq(\zw),$$
and the characteristic distribution $\zq(C_\zm)$ of the field of hyperplanes $C_\zm=C\cap\sT M_\zm$ on $M_\zm$ is $\sT\zt\left(\zq(\zw_{[\zm]})\right)$ and equals
$$ \zq(C_\zm)=(\g^0_{[\zm]})^c+\zq(C).$$

\mn In particular, if the distribution $\hat\g^0_{[\zm]}+\zq(\zw)$ (equivalently, $(\g^0_{[\zm]})^c+\zq(C)$) is of constant rank on $P_{[\zm]}$ (resp., $M_\zm$), then $\big(P_{[\zm]},\zt,M_\zm,h,\zw_{[\zm]}\big)$ is a presymplectic $\Rt$-bundle covering the precontact manifold $(M_\zm,C_\zm)$.
\end{theorem}
\begin{remark} In \cite[Section 5]{deLucas:2025}, the authors observed that $G_{[\zm]}$ can be bigger  than $G_\zm$, so our reduction in \cite{Grabowska:2023} is valid under the additional assumption $G_{[\zm]}=G_\zm$. The latter is valid, e.g., for compact $G$.
\end{remark}
\no Under the assumptions of Theorem \ref{m1}, we can carry out the presymplectic reduction (Theorem  \ref{main1}), provided that the action of $G_{[\zm]}^0$ on $M_\zm$ is free and proper. In this case, the orbits of the subgroup $G^0_{[\zm]}$ on $M_\zm$ (resp., on $P_{[\zm]}$) form a regular simple foliation, spanned by $(\g^0_{[\zm]})^c$ (resp., $\hat\g^0_{[\zm]}$). The characteristic foliation $\zq(C_\zm)$ (resp., $\zq(\zw_{[\zm]})$) is then of constant rank if and only if the intersection $(\g^0_{[\zm]})^c\cap\zq(C)$ (equivalently, $\hat\g^0_{[\zm]}\cap\zq(\zw)$) is of constant rank, e.g., $\zq(C)\subset (\g^0_{[\zm]})^c$ (resp., $\zq(\zw)\subset\hat\g^0_{[\zm]}$). The latter condition is automatically satisfied for $C$ being contact (resp., $\zw$ being symplectic).
Therefore, we can formulate the following precontact/presymplectic version of the Marsden-Weinstein-Meyer theorem.
\begin{theorem}\label{MAIN}  Let $(P,\zt,M,h,\zw)$ be a presymplectic cover of a precontact manifold $(M,C)$, let $\zr:G\ti M\to M$ be an action on $M$ of a Lie group $G$ by contactomorphisms, and let $\wt\zr:G\ti P\to P$ be the Hamiltonian cover of this action, associating with every element $\zx$ of the Lie algebra $\g$ of $G$ the contact vector field $\zx^c$ on $(M,C)$ and the Hamiltonian vector field $\hzx$ on $(P,\zw)$.

\mn Let $J:P\to\g^*$ be the corresponding contact moment map, and let $\zm\in\g^*$ be a weakly regular value of $J$, so that $P_{[\zm]}=J^{-1}([\zm]^\ti)$ is an $\Rt$-subbundle of $P$, covering a submanifold $M_\zm=\zt(P_{[\zm]})$ of $M$.  Then the connected Lie subgroup $G_{[\zm]}^0$ of $G$, corresponding to the Lie subalgebra
$$\g_{[\zm]}^0=\{\zx\in \ker(\zm)\,\big|\,\ad^*_\zx(\zm)\in[\zm]\}$$
of $\g$, acts on $P_{[\zm]}$ and $M_\zm$.

\mn Suppose that the $G^0_{[\zm]}$-action on $M_\zm$ is free and proper, so that we have a canonical submersion $\zp:M_\zm\to M(\zm)$ onto the orbit manifold $M(\zm)=M_\zm/G^0_{[\zm]}$. Suppose additionally that the distribution $(\g^0_{[\zm]})^c\cap\zq(C)$ (equivalently, $\hat\g^0_{[\zm]}\cap\zq(\zw)$) is regular (e.g, $C$ is contact/$\zw$ is symplectic).

\mn Then $M(\zm)$ is canonically a precontact manifold with the precontact structure $C(\zm)=\sT\zp(C\cap\sT M_\zm)$. Moreover, the $G^0_{[\zm]}$-action on $P_{[\zm]}$ is also free and proper, so that we have a canonical submersion  $p:P_{[\zm]}\to P(\zm)$ onto the orbit manifold $P(\zm)=P_{[\zm]}/G^0_{[\zm]}$, and $P(\zm)$ carries a presymplectic $\Rt$-bundle structure
$(P(\zm),\zt_\zm,M(\zm),h^\zm,\zw(\zm))$ such that
$$\xymatrix@C+50pt{
P_{[\zm]} \ar[r]^{p}\ar[d]_{\zt} & P(\zm)\ar[d]^{\zt_\zm} \\
M_\zm \ar[r]^{\zp} & M(\zm)}
$$
is a morphism of $\Rt$-bundles, and $p^*(\zw(\zm))=\zw\,\big|_{P_{[\zm]}}$.
Moreover, $C(\zm)$ is a contact structure (resp., $\zw(\zm)$ is symplectic) if only $C$ is contact (resp., $\zw$ is symplectic).
\end{theorem}
\no An obvious consequence of the above theorem in the case when $(M,C)$ is a contact manifold, i.e., a contact Marsden-Weinstein-Meyer theorem, is Theorem \ref{Cred}. A particular case of the latter is Theorem 1 in \cite{Willet:2002}. Note that if the action of $G_{[\zm]}$ is free and proper, then the action of $G_{[\zm]}^0$ is free and proper if and only if $G^0_{[\zm]}$ is closed in $G_{[\zm]}$ ($G_{[\zm]}$ is always a closed subgroup of $G$), thus closed in $G$. A characterisation of closeness of $G_{[\zm]}^0$ for compact $G$ is given in the following proposition.
\begin{proposition}
If $G$ is compact and connected, then $G_{[\zm]}=G_\zm$, and $G_\zm^0$ is a closed (thus compact) subgroup in $G_\zm$ if and only if $a\zm$ induces a character on $G_\zm$ for some $a\ne 0$, i.e., the linear map $a\zm:\g_\zm\to\R$ gives rise to a group homomorphism
$$\zf_{a\zm}:G_\zm\to S^1=\R/\Z,\quad \zf_{a\zm}(\exp(v))=a\la\zm,v\ran\mod\Z$$
for all $v\in\g_\zm$. In other words, $a\zm\in\g^*$ takes integer values on the kernel of $\exp:\g_\zm\to G_\zm$.
\end{proposition}
\begin{proof}
Writing $\Ad^*_g(\zm)=s_g\cdot\zm$ for $g\in G_{[\zm]}$, we see that $G_{[\zm]}\ni g\mapsto s_g\in\Rt$ is a group homomorphism. But if $G$ is compact, the corresponding subgroup in $\Rt$ is compact, thus contained in $\{\pm 1\}$.
Hence, $s_g=1$ for all $g\in G_{[\zm]}$, and $G_{[\zm]}=G_\zm$ if $G$ is additionally connected.

\mn Since for $G$ being compact and connected the subgroup $G_\zm$ is always connected (see e.g, \cite{Borel:1954,Filippini:1995}), the normal subgroup $G_\zm^0$ of $G_\zm$ is closed in $G_\zm$ if and only if $G_\zm/G_\zm^0\simeq S^1$. Let $\zf:G_\zm\to S^1=G_\zm/G_\zm^0$ be the canonical group homomorphism. Since $\g_\zm^0=\ker(\zm)$, the corresponding morphism $D\zf:\g_\zm\to\g_\zm/\g_\zm^0\simeq\R$ has the form $D\zf(v)=a\la\zm,v\ran$ for some $a\in\R$, $a\ne 0$, so $\zf=\zf_{a\zm}$.
\end{proof}
\begin{example} A simple numerical example for the above contact Marsden-Weinstein-Meyer reduction is the following. Consider $M=\sT^*\R\ti\R$ with the contact structure  determined by the Darboux contact form $\zh=\xd z-p\xd q$. In other words, this is the canonical contact structure on the first jet bundle $\sJ^1(\R;\R)$. Hence, $\zt:P=M\ti\Rt\to M$ is the trivial $\Rt$-bundle and
$$\zw=\xd s\we(\xd z-p\,\xd q)+s\cdot\xd q\we\xd p\,.$$
Moreover, $\n=s\,\pa_s$ and $\zvy=s\cdot\zh=s\cdot(\xd z-p\,\xd q)$.
Consider the 1-dimensional Lie algebra $\g=\R$ and its realization in vector fields on $M$
given by $\zx\mapsto \zx^c=\zx\cdot(\pa_q-p\,\pa_p-z\,\pa_z)$. This is the infinitesimal part of the action of the group $G=\R$ by $t.(z,p,q)=(e^{-t}z,e^{-t}p,q+t)$ on $M$ by contactomorphisms. The vector fields $\zx^c$ are contact vector fields, $\Ll_{\zx^c}\,\zh=0$, and the corresponding 1-homogeneous Hamiltonians read
$H_\zx=\zx\cdot s\cdot(p+z)$. Indeed, the Hamiltonian vector field $\hzx=X_{H_\zx}$ on $P$ reads
$$\hzx=\zx\cdot(\pa_q-p\,\pa_p-z\,\pa_z+s\,\pa_s)\,,$$
so its  projection onto $M$ is exactly $\zx^c$. The moment map $J:P\to\R^*$ satisfies
$$\zx\cdot J(s,z,q,p)=H_\zx(s,z,q,p)=\zx\cdot s\cdot(p+z),$$
so $J(s,z,q,p)=s\cdot(p+z)$ with every $\zm\in\R^*$ as a regular value.

\mn Let us consider first the case $\zm=0$, so that $P_0=J^{-1}(0)$ is the 3-dimensional and $\Rt$-invariant submanifold in $P$, defined by the equation $p+z=0$. We can take $(s,z,q)$ as global coordinates parameterizing $P_0$, $(s,z,q)\mapsto(s,z,q,-z)$, in which
\be\label{red2}\zw\,\big|_{P_0}(s,z,q)=\xd s\we(\xd z+z\,\xd q)-s\,\xd q\we\xd z\,.\ee
The submanifold $M_0=\zt(P_0)$ in $M$ is of dimension 2 and it is defined by the same equation $p+z=0$, so that $M_0$ is a transversal submanifold and $(z,q)$ serve as coordinates in $M_0$.
The rank of $\zw\,\big|_{P_0}$ is clearly 2, so the rank of the distribution $D=\ker(\zw\,\big|_{P_0})$ is 1. The distribution $D$ is thus generated by the vector field $X=\pa_q-z\,\pa_z+s\,\pa_s$, as $X$ is non-vanishing and $i_X(\zw\,\big|_{P_0})=0$. Note that $X$ is the restriction of the Hamiltonian vector field $\hat{1}=\pa_q-p\,\pa_p-z\,\pa_z+s\,\pa_s$ to $P_0$. This vector field is nowhere vanishing, so it spans $D$, and the corresponding $\R$-action on $P_0$ is $$t.(z,q,s)=(e^{-t}z,q+t,e^ts)\,.$$
Being 1-dimensional submanifolds, the trajectories of this action are described by the system of equations $s\,e^{-q}=s'$ and $z\,e^q=z'$, where  $(z',s')\in\R\ti\Rt$, so that $(z',s')$ are coordinates in $P(0)=P_0/\R$. The reduced principal $\Rt$-action  on $P(0)$ is clearly $s.(z',s')=(z',s\,s')$, and the reduced symplectic form reads (cf. (\ref{red2}))
$$\zw(0)=\xd (s'\,e^{q})\we\left(\xd(z'\,e^{-q})+z'\,e^{-q}\,\xd q\right)-s'\,e^{q}\,\xd q\we \xd(z'\,e^{-q})=\xd s'\we\xd z'.$$
The trivial $\Rt$-bundle $P(0)=\R\ti\Rt$ with the symplectic form $\zw(0)$ is the symplectic cover of  the unique (trivial) contact structure $C=\R\ti\{ 0\}\subset\sT\R$ on $\R$, which is associated with the contact form $\zh'=\xd z'$. Note that the analog of the above procedure can be done for $\sT^*\R^{n+1}\ti\R$ equipped with the  canonical contact form. If we view $\sT^*\R^{n+1}\ti\R$ as $\sT^*\R\ti\sT^*\R^n\ti\R$  with coordinates $(z,q,p,q^i,p_j)$, then the formally the same contact vector field $\hzx$ can be used as the infinitesimal part of the $\R$ action involving only coordinates $(z,q,p,s)$. The whole procedure will give the canonical contact structure on $\sT^*\R^n\ti\R$ as the reduced structure.

\mn Our procedure is intrinsic, so it does not depend on the particular choice of the (local) contact form  generating the contact structure. For instance, if we  start in  our example from  $$\tilde\zh=f\cdot\zh=f\cdot(\xd z-p\,\xd q-p_i\xd q^i)\,,$$
where $f$ is a nowhere vanishing function on $M=\sT^*\R^{n+1}\ti\R$, then we have another trivialization of $P$ with $\tilde s=f\cdot s$ and the rest of coordinates unchanged. In these coordinates the procedure is the same, and we end up finally with another trivialization of $P(0)$ with coordinates $(z',q^i,p_i,\tilde s')$, where  $\tilde s'=f(z,0,-z,q^i,p_j)\cdot s'$. This corresponds to the new contact form
$$\tilde\zh'=f(z,0,-z,q^i,p_j)\cdot\zh(z,q,p,q^i,p_j),$$
which is clearly equivalent to $\zh$, so that it induces the same reduced contact structure.

\mn If now $\zm\in\R^*$, $\zm\ne 0$, then $P_\zm$ is the submanifold in $P=M\ti\Rt$ defined by the equation $s=\zm/(p+z)$, $p+z\ne 0$, and $P'=P_{[\zm]}=J^{-1}((\R^*)^\ti)$ is an open-dense submanifold in $P$ of those $(z,p,q,s)\in\sT^*\R\ti\R\ti\Rt$ such that $(p+z)\ne 0$. The action of $\R$ on $P$ preserves $P'$, so its action on $M$ preserves the open-dense submanifold $M'=\zt(P')$
defined by $(p+z)\ne 0$. The group $G=\R$ is commutative, so $\g_{[\zm]}=\g_\zm=\g$ and $\g_\zm^0=\{ 0\}$. Hence, $M'$ and $P'$ with restricted contact (resp., symplectic) structures are the reduced structures in this case.
\end{example}
\no In the next example, we consider first jet prolongations of group actions on line bundles, which can serve as canonical examples of group actions by contactomorphisms on nontrivial contact manifolds.
\begin{example} In this example, we show how to produce a contact analogue of the Marsden-Weinstein-Meyer reduction for the cotangent lifts of group actions on a manifold (see e.g, \cite{Abraham:1978,Ortega:2004,Reyes:2005}).

Let $\zt_0:L\to Q$ be a line bundle, $\zp_0:L^*\to Q$ be its dual, and $G$ be a Lie group. Note first that, similarly to the case of cotangent bundles, any $G$-action $\hat\zr:G\ti L\to L$ on $L$ by vector bundle automorphisms can be canonically lifted to a contact $G$-action $\zr:G\ti \sJ^1L^*\to\sJ^1L^*$ on the first jet bundle $\sJ^1L^*$ of sections $\zs:Q\to L^*$ of $L^*$, equipped with its canonical contact structure. Indeed, if $\hat\zr^*:G\ti L^*\to L^*$ is the dual action on $L^*$ and the vector bundle isomorphisms $\hat\zr_g$ reads
$$\xymatrix@C+50pt{
L \ar[r]^{{\hat\zr}_g}\ar[d]_{\zt_0} & L\ar[d]^{\zt_0} \\
Q \ar[r]^{{\hat\zr}^0_g} & Q\,,}
$$
then $\zr$ is uniquely determined by
$$\zr_g(\sj^1_q(\zs))=\sj^1_{\hat\zr^0_g(q)}(\hat\zr^*_g\circ\zs\circ {\hat\zr}^0_{g^{-1}} ).$$
Recall that infinitesimal automorphisms of a vector bundle consist of linear vector fields on the bundle and the symplectic cover of the contact manifold $\sJ^1L^*$ is $P=\sT^*\Lt$, with the phase lift of the $\Rt$-action on $\Lt$ and the canonical symplectic structure of the cotangent bundle \cite{Grabowska:2022,Grabowski:2013}. The Hamiltonian action of $G$ on $\sT^*\Lt$, lifted from the contact action on $\sJ^1 L^*$, is the standard cotangent lift of the $G$ action on $L$, thus on $\Lt$.

If $\zx\in\g$ is an element in the Lie algebra $\g$ of $G$ and $\hzx_0$ is the corresponding linear vector field on $L$ (thus a 1-homogeneous vector field on $\Lt$), associated with the $G$-action $\hat\zr$ on $L$ by vector bundle automorphisms, then the contact Hamiltonian $H_\zx$ on $\sT^*\Lt$ is $H_\zx=\zi_{\hzx_0}$, where $\zi_{\hzx_0}$ is the linear function on $\sT^*L$, thus on $\sT^*\Lt$, associated with the vector field $\hzx_0$ on $L$. Since $\hzx_0$ is a linear vector field on $L$, it is easy to see that the function $\zi_{\hzx_0}$ on $\sT^*\Lt$ is additionally 1-homogeneous with respect to the lifted principal $\Rt$-bundle structure, thus a contact Hamiltonian. The corresponding Hamiltonian vector field $\hzx=X_{H_\zx}$ is therefore the well-known cotangent lift of the vector field $\hzx_0$ on $L$, and the contact moment map reads
$$J:\sT^*\Lt\to\g^*,\quad \La J(\za_v),\zx\Ra=\zi_{\hzx_0}(\za_v)=\La\za_v,\hzx_0(v)\Ra,$$
where $\za_v\in\sT^*_v\Lt$, $v\in\Lt$.

\mn For $0=\zm\in\g^*$ we can proceed now with the traditional Marsden-Weinstein-Meyer reduction, which is known to yield the reduced symplectic manifold $J^{-1}(0)/G=\sT^*(\Lt/G)$.
Since clearly $\Lt/G=(L/G)^\ti$, the corresponding reduced contact manifold is $\sJ^1(L/G)^*=\sJ^1(L^*/G)$ with its canonical contact structure. The case $\zm\ne 0$ is more complicated even in the standard situation and leads to the so-called \emph{fibration cotangent bundle reduction} \cite[Theorem 6.6.8]{Ortega:2004}. The studies on corresponding contact analogs of the latter we postpone to a separate paper.

\end{example}
\section{Conclusions and outlook}
We presented in this paper a fully intrinsic geometric approach to reductions of contact manifolds, which serves for general (also nontrivial) contact structures. This approach is closely related to symplectic reductions, due to a one-to-one correspondence between contact manifolds and symplectic manifolds of a special type (symplectic $\Rt$-bundles). Actually, all tools can be easily adapted to the precontact/presymplectic structures introduced in the paper. We obtained precontact-to-contact reductions, as well as precontact versions of the celebrated Marsden-Weinstein-Meyer theorem. The presentation was concentrated on introducing geometric concepts and tools. The next step will be devoted to a closer analysis of important examples coming from physics and other potential applications.

The precontact/presymplectic framework can be crucial for these aims, since presymplectic structures play a fundamental r\^ole in various problems of physics origins, especially related to constrained systems (e.g, \cite{Cantrijn:1999,Gotay:1978,Reyes:2005,Rothe:2010}) or time-dependent mechanics \cite{Guzman:2010}. For instance, the presymplectic approach is a crucial step in the relativistic theory, e.g, the equations of motion of a charged spinning particle moving in a space-time (with or without torsion) in the presence of an electromagnetic field. The power of presymplectic formulation resides in the fact that it is manifestly covariant and does not require non-relativistic concepts, like absolute time, to describe dynamics. Also, the space of solutions of first-order Hamiltonian field theories is a presymplectic
manifold, which was exploited in a series of papers \cite{Ciaglia:2020,Ciaglia:2020a,Ciaglia:2022}
describing covariant variational evolution and Poisson/Jacobi brackets on the space of functions on the solutions to a variational problem. We plan to carry out studies in these directions and publish them in forthcoming papers.

\section*{Acknowledgments}
The authors thank Witold Respondek for his useful comments on the classification of hyperplane fields. We also thank Javier de Lucas for pointing out the reference \cite{deLucas:2025}.

\small{\vskip.5cm}

\noindent Katarzyna GRABOWSKA\\
Faculty of Physics \\
University of Warsaw\\
Pasteura 5,
02-093 Warszawa, Poland
\\Email: konieczn@fuw.edu.pl\\
https://orcid.org/0000-0003-2805-1849\\

\noindent Janusz GRABOWSKI\\ Institute of
Mathematics\\  Polish Academy of Sciences\\ \'Sniadeckich 8, 00-656 Warszawa, Poland
\\Email: jagrab@impan.pl \\
https://orcid.org/0000-0001-8715-2370\\


\begin{thebibliography}{V}
\bibitem{Abraham:1978} R.~Abraham, J.~E.~Marsden,
\newblock{Foundations of Mechanics,}
\newblock{Benjamin-Cummings, New York, 2nd edition, 1978.}

\bibitem{Albert:1998} C.~Albert,
\newblock{Le th\'eor\`eme de r\'eduction de Marsden--Weinstein en g\'eom\'etrie cosymplectique et de contact,}
\newblock{\emph{J. Geom. Phys.} \textbf{6} (1989), 627--649.}

\bibitem{Arnold:1989} V.~I.~Arnold,
\newblock{\emph{Mathematical Methods of Classical Mechanics},}
\newblock{Springer, 2nd edition, 1989.}

\bibitem{Benenti:1983} S.~Benenti,
\newblock{The category of symplectic reductions,}
\newblock{\emph{Proceedings of the international meeting on geometry and physics} (Florence, 1982), 11--41, Pitagora, Bologna, 1983.}

\bibitem{Benenti:1982} S.~Benenti, W.~M.~Tulczyjew,
\newblock{Remarques sur les r\'eductions symplectiques,}
\newblock{\emph{C. R. Acad. Sci. Paris Sér. I Math.} \textbf{294} (1982), 561--564.}

\bibitem{Blankenstein:2001} G.~Blankenstein, A.~J.~van der Schaft,
\newblock{Symmetry and reduction in implicit generalized Hamiltonian systems,}
\newblock{\emph{Rep. Math. Phys.} \textbf{47} (2001), 57--100.}


\bibitem{Boyer:2000} C.~P.~Boyer, K.~Galicki,
\newblock{A note on toric contact geometry,}
\newblock{\emph{J. Geom. Phys.} \textbf{35} (2000), 288--298.}

\bibitem{Borel:1954} A.~Borel,
\newblock{K\"ahlerian coset spaces of semisimple Lie groups,}
\newblock{\emph{Proc. Nat. Acad. Sci. U.S.A.} \textbf{40} (1954), 1147--1151.}

\bibitem{Brahic:2014} O.~Brahic, R.-L.~Fernandes,
\newblock{Integrability and reduction of Hamiltonian actions on Dirac manifolds,}
\newblock{\emph{Indag. Math. (N.S.)} \textbf{25} (2014), 901--925.}

\bibitem{Bravetti:2017a} A.~Bravetti, H.~Cruz, D.~Tapias,
\newblock{Contact {H}amiltonian mechanics,}
\newblock{\emph{Ann. Phys.} \textbf{376} (2017), 17--39.}

\bibitem{Bruce:2016} A.~J.~Bruce, K.~Grabowska, J.~Grabowski,
\newblock{Linear duals of graded bundles and higher analogues of (Lie) algebroids},
\newblock{\emph{J. Geom. Phys.} \textbf{101 }(2016), 71--99.}

\bibitem{Bruce:2017} A.~J.~Bruce, K.~Grabowska, J.~Grabowski,
\newblock{Remarks on contact and Jacobi geometry,}
\newblock{\emph{SIGMA Symmetry Integrability Geom. Methods Appl.} \textbf{13} (2017), Paper No. 059, 22 pp.}

\bibitem{Bursztyn:2012} H.~Bursztyn, A.~Cabrera,
\newblock{Symmetries and reduction of multiplicative 2-forms,}
\newblock{\emph{J. Geom. Mech.} \textbf{4} (2012), 111--127.}

\bibitem{Cantrijn:1999} F.~Cantrijn, M.~de~Le\'on, J.~C.~Marrero, D~Mart\'{\i}n~de~Diego,
\newblock{Reduction of constrained systems with symmetries,}
\newblock{\emph{J. Math. Phys.} \textbf{40} (1999), 795--820.}

\bibitem{Ciaglia:2018} F.~M. Ciaglia, H.~Cruz, G.~Marmo,
\newblock{Contact manifolds and dissipation, classical and quantum,}
\newblock{\emph{Ann. Phys.} \textbf{398} (2018), 159--179.}

\bibitem{Ciaglia:2020} F.~M.~Ciaglia, F.~Di~Cosmo, A.~Ibort, G.~Marmo, L.~Schiavone,
\newblock{Covariant variational evolution and Jacobi brackets: Particles,}
\newblock{\emph{Mod. Phys. Lett. A} \textbf{35} (2020), 2020001.}

\bibitem{Ciaglia:2020a} F.~M.~Ciaglia, F.~Di~Cosmo, A.~Ibort, G.~Marmo, L.~Schiavone,
\newblock{Covariant variational evolution and Jacobi brackets: Fields,}
\newblock{\emph{Mod. Phys. Lett. A} \textbf{35} (2020), 2050206.}

\bibitem{Ciaglia:2022} F.~M.~Ciaglia, F.~Di~Cosmo, A.~Ibort, G.~Marmo, L.~Schiavone, A.~Zampini,
\newblock{Symmetries and covariant Poisson brackets on presymplectic manifolds,}
\newblock{\emph{Symmetry} \textbf{14} (2022), 70.}

\bibitem{Cruz:2018} H.~Cruz,
\newblock{Contact Hamiltonian mechanics. An extension of symplectic Hamiltonian mechanics,}
\newblock{\emph{J. Phys.: Conference Series} \textbf{1071} (2018), 012010.}

\bibitem{Darboux:1882} G.~Darboux,
\newblock{Sur le probl\'eme de Pfaff, I, II,}
\newblock{\emph{Bull. Sci. Math. $2^e$ s\'erie} \textbf{6} (1882), 14--36 and 49--68.}

\bibitem{Dazord:1991}  P.~Dazord, A.~Lichnerowicz, Ch.-M.~Marle,
\newblock{Structure locale des vari\'et\'es de Jacobi,}
\newblock{\emph{ J.Math. Pures et Appl.}, \textbf{70} (1991), 101--152.}

\bibitem{Dragulete:2002} O.~Dr\v agulete, L.~Ornea, T.~S.~Ratiu,
\newblock{Cosphere bundle reduction in contact geometry,}
\newblock{\emph{J. Symplectic Geom.} \textbf{1} (2003), 695--714.}

\bibitem{deLeon:2020} M.~de~Le\'on, J.~Gaset, M.~Lainz Valc\'azar, X.~Rivas, N.~Rom\'an-Roy,
\newblock{Unified Lagrangian-Hamiltonian formalism for contact systems,}
\newblock{\emph{Fortschr. Phys.} \textbf{68} (2020), 2000045, 12 pp.}

\bibitem{deLeon:2022} M.~de~Le\'on, J.~Gaset, X.~Gr\`acia, M.~C.~Mu\~noz-Lecanda, X.~Rivas,
\newblock{Time-dependent contact mechanics,}
\newblock{\emph{Monatshefte f\"ur Mathematik}, https://doi.org/10.1007/s00605-022-01767-1.}

\bibitem{deLeon:2019} M.~de~Le\'on, M.~Lainz Valc\'azar,
\newblock{Contact Hamiltonian systems,}
\newblock{\emph{J. Math. Phys.} \textbf{60} (2019), 102902, 18pp.}

\bibitem{deLeon:2019a} M.~de~Le\'on and M.~Lainz Valc\'azar,
\newblock{Singular Lagrangians and precontact Hamiltonian systems,}
\newblock{\emph{Int. J. Geom. Methods Mod. Phys.} \textbf{16} (2019), 1950158, 39 pp.}

\bibitem{deLucas:2025} J.~de~Lucas, X.~Rivas, S.~Vilari\~no, B.~M.~Zawora,
\newblock{Reduction of k-contact field theories,}
\newblock{\emph{arXiv:2505.05462}.}

\bibitem{Cosmo:2025} F.~Di~Cosmo, K.~Grabowska, J.~Grabowski,
\newblock{Jacobi algebroids and Jacobi sigma models,}
\newblock{\emph{Rev. Math. Phys.} (2025) (to appear).}

\bibitem{Munoz:1999} A.~Echeverr\'{\i}a-Enr\'{\i}quez, M.~C.~Mu\~noz-Lecanda, N.~Rom\'an-Roy,
\newblock{Reduction of presymplectic manifolds with symmetry,}
\newblock{\emph{Rev. Math. Phys.} \textbf{11} (1999), 1209--1247.}

\bibitem{Esen:2021} O.~Esen, M.~Lainz~Valc\'azar, M.~de~Le\'on, J.~C.~Marrero,
\newblock{Contact Dynamics: Legendrian and Lagrangian Submanifolds,}
\newblock{\emph{Mathematics} \textbf{9} (2021), 9.}

\bibitem{Filippini:1995} R.~J.~Filippini,
\newblock{The symplectic geometry of the theorems of Borel-Weil and Peter-Weyl,}
\newblock{\emph{Thesis}, University of California, Berkeley, 1995.}

\bibitem{Gaset:2020} J.~Gaset, X.~Gr\`acia, M.~C.~Mu\~noz-Lecanda, X.~Rivas, N.~Rom\'an-Roy,
\newblock{New contributions to the Hamiltonian and Lagrangian contact formalisms for dissipative
mechanical systems and their symmetries,}
\newblock{\emph{Int. J. Geom. Methods Mod. Phys.} \textbf{17} (2020), 2050090, 27 pp.}

\bibitem{Geiges:1997} H.~Geiges,
\newblock{Constructions of contact manifolds,}
\newblock{\emph{Math. Proc. Cambridge Philos. Soc.} \textbf{121} (1997), 455--464.}

\bibitem{Geiges:2001} H.~Geiges,
\newblock{A brief history of contact geometry and topology,}
\newblock{\emph{Expo. Math.} \textbf{19} (2001), 25--53.}

\bibitem{Geiges:2008} H.~Geiges,
\newblock{An introduction to contact topology,}
\newblock{\emph{Cambridge Studies in Advanced Mathematics} \textbf{109}.}
\newblock{Cambridge University Press, Cambridge, 2008.}

\bibitem{Gotay:1978} M.~J.~Gotay, J.~M.~Nester, G.~Hinds,
\newblock{Presymplectic manifolds and the Dirac-Bergmann theory of constraints,}
\newblock{\emph{J. Math. Phys.} \textbf{19} (1978), 2388--2399.}


\bibitem{Grabowska:2022} K.~Grabowska, J.~Grabowski,
\newblock{A novel approach to contact Hamiltonians and contact Hamilton-Jacobi Theory.}
\newblock{\emph{J. Phys. A} \textbf{55} (2022), 435204 (34pp).}

\bibitem{Grabowska:2023} K.~Grabowska, J.~Grabowski,
\newblock{Reductions: precontact versus presymplectic,}
\newblock{\emph{Ann. Mat. Pura Appl.} \textbf{202} (2023), 2803--2839.}

\bibitem{Grabowska:2024} K.~Grabowska, J.~Grabowski,
\newblock{Contact geometric mechanics: the Tulczyjew triples,}
\newblock{\emph{Adv. Theor. Math. Phys.} \textbf{28} (2024), 599--654.}

\bibitem{Grabowska:2024a} K.~Grabowska, J.~Grabowski, M.~Ku\'s, G.~Marmo,
\newblock{Contactifications: a Lagrangian description of compact Hamiltonian systems,}
\newblock{\emph{J. Phys. A} \textbf{57} (2024), 395204 (31pp).}

\bibitem{Grabowski:2003a} J.~Grabowski,
\newblock{Quasi-derivations and QD-algebroids,}
\newblock{\emph{Rep. Math. Phys.} {\bf 32} (2003), 445--451.}

\bibitem{Grabowski:2007} J.~Grabowski,
\newblock{Local Lie algebra determines base manifold,}
\newblock{From geometry to quantum mechanics, 131-145, \emph{Progr. Math.} \textbf{252}, Birkhäuser Boston, Boston, MA, 2007.}

\bibitem{Grabowski:2013} J.~Grabowski,
\newblock{Graded contact manifolds and contact Courant algebroids}
\newblock{\emph{J. Geom. Phys. } \textbf{68} (2013), 27--58.}

\bibitem{Grabowski:2013a} J.~Grabowski,
\newblock{Brackets,}
\newblock{\emph{Int. J. Geom. Methods Mod. Phys.} \textbf{10} (2013), 1360001, 45 pp.}

\bibitem{Grabowski:2004} J.~Grabowski, D.~Iglesias, J.~C.~Marrero, E.~Padr\'on, P.~Urba\'nski,
\newblock{Poisson-Jacobi reduction of homogeneous tensors,}
\newblock{\emph{J. Phys. A} \textbf{37} (2004), 5383--5399.}

\bibitem{Grabowski:1994} J.~Grabowski, G.~Landi, G.~Marmo, G.~Vilasi,
\newblock{ Generalized reduction procedure: symplectic and Poisson formalism,}
\newblock{\emph{Forts. Phys.} \textbf{42} (1994), 393--427.}

\bibitem{Grabowski:2003} J.~Grabowski, G.~Marmo,
\newblock{The graded Jacobi algebras and (co)homology,}
\newblock{\emph{J. Phys. A} \textbf{36} (2003), 161--181.}

\bibitem{Grabowski:2009} J.~Grabowski, M.~Rotkiewicz,
\newblock{Higher vector bundles and multi-graded symplectic manifolds,}
\newblock{\emph{J. Geom. Phys.} \textbf{59} (2009), 1285--1305.}

\bibitem{Grmela:2014} M.~Grmela,
\newblock{Contact geometry of mesoscopic thermodynamics and dynamics,}
\newblock{\emph{Entropy} \textbf{16} (2014), 1652--1686.}

\bibitem{Guedira:1984} F.~Gu\'edira,  A.~Lichnerowicz,
\newblock{G\'eom\'etrie des alg\'ebres de Lie locales de Kirillov,}
\newblock{\emph{J. Math. Pures Appl.}, \textbf{63} (1984), 407--484.}

\bibitem{Guillemin:1982} V.~Guillemin, S.~Sternberg,
\newblock{Homogeneous quantization and multiplicities of group representations,}
\newblock{\emph{J. Functional Analysis} \textbf{47} (1982), 344--380.}

\bibitem{Guzman:2010} E.~Guzm\'an, J.~C.~Marrero,
\newblock{Time-dependent mechanics and Lagrangian submanifolds of presymplectic and Poisson manifolds,}
\newblock{\emph{J. Phys. A} \textbf{43} (2010), 505201, 23 pp.}

\bibitem{Ibort:1997} A.~Ibort, M.~de~Le\'on, G.~Marmo,
\newblock{Reduction of Jacobi manifolds,}
\newblock{\emph{J. Phys. A} \textbf{30} (1997), 2783--2798.}

\bibitem{Kirillov:1976} A.~A.~Kirillov,
\newblock{Local Lie algebras,}
\newblock{\emph{Russian Math. Surveys} \textbf{31} (1976), no. 4, 55--75.}

\bibitem{Le:2017} H\^ong V\^an L\^e, A.~G.~Tortorella, L.~Vitagliano,
\newblock{Jacobi bundles and the BFV-complex,}
\newblock{\emph{J. Geom. Phys.} \textbf{121} (2017), 347--377.}

\bibitem{Le:2018} H\^ong V\^an L\^e, Yong-Geun Oh, A.~G.~Tortorella, L.~Vitagliano,
\newblock{Deformations of coisotropic submanifolds in Jacobi manifolds,}
\newblock{\emph{J. Symplectic Geom.} \textbf{16} (2018), 1051--1116.}

\bibitem{Lerman:2001} E.~Lerman, C.~Willett,
\newblock{The topological structure of contact and symplectic quotients,}
\newblock{\emph{Internat. Math. Res. Notices} \textbf{2001} (2001), 33--52.}

\bibitem{Lerman:2003} E.~Lerman,
\newblock{Geodesic flows and contact toric manifolds,}
\newblock{\emph{Symplectic geometry of integrable Hamiltonian systems} (Barcelona, 2001), 175--225, Adv. Courses Math. CRM Barcelona, Birkhäuser, Basel, 2003.}

\bibitem{Lerman:2004} E.~Lerman,
\newblock{Contact fiber bundles,}
\newblock{\emph{J. Geom. Phys.} \textbf{49} (2004), 52--66.}

\bibitem{Libermann:1987} P.~Libermann, C.-M.~Marle,
\newblock{Symplectic Geometry and Analytical Mechanics,}
\newblock{\emph{Mathematics and its Applications} \textbf{35},
D. Reidel Publishing Co., Dordrecht, 1987.}

\bibitem{Lichnerowicz:1978} A.~Lichnerowicz,
\newblock{Les vari\'et\'es de Jacobi  et leurs alg\'ebres de Lie associ\'ees,}
\newblock{\emph{J. Math. Pures Appl.}, \textbf{57} (1978), 453--488.}

\bibitem{Loose:2001} F.~Loose,
\newblock{Reduction in contact geometry,}
\newblock{\emph{J. Lie Theory} \textbf{11} (2001), 9--22.}

\bibitem{Marle:1991} C.~M.~Marle,
\newblock{On Jacobi manifolds and Jacobi bundles,}
\newblock{in Symplectic geometry, groupoids, and integrable systems (Berkeley, CA, 1989), 227--246, \emph{Math. Sci. Res. Inst. Publ.} \textbf{20}, Springer, New York, 1991.}

\bibitem{Marsden:1974} J.~E.~Marsden, A.~Weinstein,
\newblock{Reduction of symplectic manifolds with symmetry,}
\newblock{\emph{Rep. Math. Phys.} \textbf{5} (1974), 121--130.}

\bibitem{Marsden:2001} J.~E.~Marsden, A.~Weinstein,
\newblock{Comments on the history, theory, and applications of symplectic reduction,}
\newblock{in Quantization of singular symplectic quotients, 1--19, \emph{Progr. Math.} \textbf{198}, Birkhäuser, Basel, 2001.}

\bibitem{Meyer:1973} K.~R.~Meyer,
\newblock{Symmetries and integrals in mechanics,}
\newblock{\emph{Dynamical systems} (Proc. Sympos., Univ. Bahia, Salvador, 1971), pp. 259--272. Academic Press, New York, 1973.}

\bibitem{Mikami:1987} K.~Mikami,
\newblock{Local Lie algebra structure and momentum mapping,}
\newblock{\emph{J. Math. Soc. Japan} \textbf{39} (1978), 233--246.}

\bibitem{Mikami:1989} K.~Mikami,
\newblock{Reduction of local Lie algebra structures,}
\newblock{\emph{Proc. Amer. Math. Soc.} \textbf{105} (1989), 686--691.}

\bibitem{Mrugala:1991} R.~Mruga\l a, J.~D. Nulton, J.~C. Sch{\"o}n, P.~Salamon,
\newblock{Contact structure in thermodynamic theory,}
\newblock{\emph{Rep. Math. Phys.} \textbf{29} (1991), 109--121.}


\bibitem{Nunes:1989} J.~M.~Nunes da Costa,
\newblock{R\'eduction des vari\'et\'es de Jacobi,}
\newblock{\emph{C. R. Acad. Sci. Paris S\'er. I Math.} \textbf{308} (1989), 101--103.}

\bibitem{Nunes:1990} J.~M.~Nunes da Costa,
\newblock{Une g\'en\'eralisation, pour les vari\'et\'es de Jacobi, du th\'eor\`eme de r\'eduction de Marsden-Weinstein,}
\newblock{\emph{C. R. Acad. Sci. Paris S\'er. I Math.} \textbf{310} (1990), 411--414.}

\bibitem{Ortega:2004} J.-P.~Ortega, T.~S.~Ratiu,
\newblock{Momentum maps and Hamiltonian reduction,}
\newblock{\emph{Progress in Mathematics} \textbf{222}, Birkhäuser Boston, Inc., Boston, MA, 2004.}

\bibitem{Palais:1961} R.~Palais,
\newblock{On the existence of slices for actions of non-compact Lie groups,}
\newblock{\emph{Ann. Math.} \textbf{73} (1961), 295--323.}

\bibitem{Rajeev:2008} S.~G.~Rajeev,
\newblock{A Hamilton--Jacobi formalism for thermodynamics,}
\newblock{\emph{Ann. Phys.} \textbf{323} (2008), 2265--2285.}

\bibitem{Reyes:2005} E.~G~Reyes,
\newblock{On the motion of particles and strings, presymplectic mechanics, and the variational bicomplex,}
\newblock{\emph{Gen. Relativity Gravitation} \textbf{37} (2005), 437--459.}

\bibitem{Rosen:2020} D.~Rosen, J.~Zhang,
\newblock{Chekanov's dichotomy in contact topology,}
\newblock{\emph{Math. Res. Lett.} \textbf{27} (2020), 1165--1193.}

\bibitem{Rothe:2010} H.~J.~Rothe, K.~D.~Rothe,
\newblock{Classical and quantum dynamics of constrained Hamiltonian systems,}
\newblock{\emph{World Scientific Lecture Notes in Physics} \textbf{81},}
\newblock{World Scientific Publishing Co. Pte. Ltd., Hackensack, NJ, 2010.}

\bibitem{Shaft:2018} A.~van~der~Schaft, B.~Maschke,
\newblock{Geometry of thermodynamic processes,}
\newblock{\emph{Entropy} \textbf{20} (2018), 925, 23 pp.}

\bibitem{Souriau:1970} J.-M.~Souriau,
\newblock{Structure des syst\`emes dynamiques,}
\newblock{Ma{\^i}trises de math\'ematiques, Dunod, Paris 1970.}

\bibitem{Tortorella:2017} A.~G.~Tortorella,
\newblock{Deformations of coisotropic submanifolds in Jacobi manifolds, PhD. thesis 2017,}
\newblock{\emph{arXiv:1705.08962}.}

\bibitem{Tortorella:2018} A.~G.~Tortorella, 
\newblock{Rigidity of integral coisotropic submanifolds of contact manifolds,}
\newblock{\emph{Lett. Math. Phys.} \textbf{108} (2018), 883--896.}

\bibitem{Vitagliano:2018} L.~Vitagliano,
\newblock{Dirac-Jacobi bundles,}
\newblock{\emph{J. Symplectic Geom.} \textbf{16} (2018), 485--561.}

\bibitem{Vitagliano:2016} L.~Vitagliano, A.~Wade,
\newblock{Holomorphic Jacobi manifolds,}
\newblock{\emph{Internat. J. Math.} \textbf{31} (2020), 2050024, 39 pp.}

\bibitem{Willet:2002} C.~Willett,
\newblock{Contact reduction,}
\newblock{\emph{Trans. Amer. Math. Soc.} \textbf{354} (2002), 4245--4260.}

\bibitem{Zambon:2006} M.~Zambon, C.~Zhu,
\newblock{Contact reduction and groupoid actions,}
\newblock{\emph{Trans. Amer. Math. Soc.} \textbf{358} (2006), 1365--1401.}

\end{thebibliography}
\end{document}